\documentclass[final,leqno]{siamltex1213}
\usepackage{subfigure}
\usepackage{amsmath}
\usepackage{algorithm}
\usepackage{algorithmic}

\title{A Preconditioned Hybrid SVD Method for Computing Accurately Singular Triplets of Large Matrices}
\author{Lingfei Wu\footnotemark[1]
  \and  Andreas Stathopoulos\footnotemark[1]}

\begin{document}
\maketitle
\renewcommand{\thefootnote}{\fnsymbol{footnote}}
\footnotetext[1]{Department of Computer Science, College of William and Mary, Williamsburg, Virginia 23187-8795, 
U.S.A.(lfwu@cs.wm.edu, andreas@cs.wm.edu)} 
\renewcommand{\thefootnote}{\arabic{footnote}}

\begin{abstract}
The computation of a few singular triplets of large, sparse matrices is a challenging task, especially when the smallest magnitude singular values are needed in high accuracy. Most recent efforts try to address this problem through variations of the Lanczos bidiagonalization method, 
but they are still challenged even for medium matrix sizes due to the difficulty of the problem.
We propose a novel SVD approach that can take advantage of preconditioning and of any well designed eigensolver to compute both largest and smallest singular triplets.
Accuracy and efficiency is achieved through a hybrid, two-stage meta-method, PHSVDS. In the first stage, PHSVDS solves the normal equations up to the best achievable accuracy. If further accuracy is required, the method switches automatically to an eigenvalue problem with the augmented matrix. Thus it combines the advantages of the two stages, faster convergence and accuracy, respectively. For the augmented matrix, solving the interior eigenvalue is facilitated by a proper use of the good initial guesses from the first stage and an efficient implementation of the refined projection method. We also discuss how to precondition PHSVDS and to cope with some issues that arise. 
Numerical experiments illustrate the efficiency and robustness of the method. 
\end{abstract}

\section{Introduction}
The Singular Value Decomposition (SVD) is a ubiquitous computational kernel in science and engineering.
Many, highly diverse applications require a few of the largest singular values of a large sparse matrix $A$ and the associated left and right singular vectors (singular triplets). 
These triplets play a critical role in compression and model reduction.
A smaller, but increasingly important, set of applications requires a few smallest singular triplets.
Examples include least squares problems, determination of matrix rank, low rank approximation, and computation of pseudospectrum
\cite{Golub1996MC,trefethen1997numerical}.
Recently we have used such techniques to reduce the variance in Monte Carlo estimations of the 
  the trace of the inverse of a large sparse matrix.

It is well known that the computation of the smallest singular triplets presents challenges both to the speed of convergence and the accuracy of iterative methods. In this paper, we mainly focus on the problem of finding the smallest singular triplets. Assume $ A \in \Re^{m \times n} $ is a large sparse matrix with full column rank and $ m \geq n $. The (economy size) singular value decomposition of A can be written as:
\begin{equation}
A = U \Sigma  V^T
\end{equation} 
where $ U = [u_{1},\ldots , u_{n}] \in \Re^{m \times n}$ is an orthonormal set of the left
  singular vectors and $ V = [v_{1},\ldots , v_{n}] \in \Re^{n \times n}$ is the unitary matrix of
  the right singular vectors. $ \Sigma = diag(\sigma_{1}, \ldots, \sigma_{n}) \in \Re^{n \times n} $  contains the singular values of $A$, $\sigma_{1} \leq \ldots \leq \sigma_{n}$. 
We will be looking for the smallest $k \ll n$ singular triplets 
	$ \{ \sigma_{i}, u_{i}, v_{i} \}, i = 1, \ldots, k$.

There are two approaches to compute the singular triplets $ \{ \sigma_{i}, u_{i}, v_{i} \} $ by using a Hermitian eigensolver. Using MATLAB notation, the first approach seeks eigenpairs of the augmented matrix $ B = [0\ A^T ; A\ 0] \in \Re^{(m+n) \times (m+n)}$, which has eigenvalues $ \pm \sigma_i$ with
 corresponding eigenvectors $([v_i;u_i], [-v_i;u_i])$, as well as $m - n$ zero eigenvalues
 \cite{golub1965calculating,golub1981block,cullum1983lanczos}. 
The main advantage of this approach is that iterative methods can potentially compute the smallest singular values accurately, i.e., with residual norm close to $O(\|A\| \epsilon_{mach})$. However, convergence of eigenvalue iterative methods is slow since it is a highly interior eigenvalue problem, and even the use of iterative refinement or inverse iteration involves a maximally indefinite matrix \cite{parlett1980symmetric}. For restarted iterative methods convergence is even slower, irregular, and often the required eigenvalues are missed since the Rayleigh-Ritz projection method does not effectively extract the appropriate information for interior eigenvectors \cite{morgan1991computing,
Morgan1998Harmonic,jia1997refined}.

The second approach computes eigenpairs of the normal equations matrix $ C = A^T A \in \Re^{n \times n} $ which has eigenvalues $ \sigma_{i}^2$ and associated eigenvectors $v_i$. 
If $\sigma_i \neq 0$, the corresponding left singular vectors are obtained as 
  $u_i = \frac{1}{\sigma_i} A^T v_i$. 
$C$ is implicitly accessed through successive matrix-vector multiplications. 
The squaring of the singular values works in favor of this approach with Krylov methods, 
  especially with largest singular values since their relative separations increase.
Although the separations of the smallest singular values become smaller, 
  we show in this paper that this approach still is faster than Krylov methods on $B$ 
  because it avoids indefiniteness.
On the other hand, squaring the matrix limits the accuracy at which smallest singular triplets 
  can be obtained.
Therefore, this approach is typically followed by a second stage of iterative 
  refinement for each needed singular triplet 
\cite{philippe1997computation,dongarra1983improving,berry1992large}. 
However, this one-by-one refinement does not exploit information from other 
  singular vectors and thus is it not as efficient as an eigensolver applied on $B$ 
  with the estimates of the first stage. 

The Lanczos bidiagonalization (LBD) method \cite{Golub1996MC, golub1965calculating} 
  is accepted as an accurate and more efficient method for seeking singular triplets 
  (especially smallest), and numerous variants have been proposed 
\cite{larsen2001combining, jia2003implicitly, kokiopoulou2004computing, baglama2005augmented, 
baglama2006restarted, baglama2013implicitly, jia2010refined}.
LBD builds the same subspace as Lanczos on matrix $C$, but since it works on 
  $A$ directly, it avoids the numerical problems of squaring. 
However, the Ritz vectors often exhibit slow, irregular convergence when the 
  smallest singular values are clustered. 
To address this problem, harmonic projection 
  \cite{kokiopoulou2004computing, baglama2005augmented},
  refined projection \cite{jia2003implicitly}, and 
  their combinations \cite{jia2010refined} have been applied to LBD. 
Despite remarkable algorithmic progress, current LBD methods are still in 
  development, with only few existing MATLAB implementations that serve 
  mainly as a testbed for mathematical research. 
We show that a two stage approach based on a well designed 
  eigenvalue code can be more robust and efficient for 
  a few singular triplets. 
Most importantly, our approach can use preconditioning through the 
  eigensolver,
  which is not directly possible with LBD but becomes 
  crucial because of the difficulty of the problem even for medium matrix sizes.

The Jacobi-Davidson type SVD method, JDSVD \cite{hochstenbach2001jacobi,hochstenbach2004harmonic}, is based 
  on an inner-outer iteration and can also use preconditioning.
It obtains the left and right singular vectors directly 
  from a projection of $B$ on two subspaces and, although it avoids the
  numerical limitations of matrix $C$, it needs a harmonic 
   \cite{morgan1991computing,paige1995approximate,sleijpen2000jacobi}  
  or a refined projection method \cite{jia1997refined,stewart2001matrix} to avoid irregular
   Rayleigh-Ritz convergence. JDSVD often has difficulty computing the smallest singular values of a 
  rectangular matrix, especially without preconditioning,  
  due to the presence of zero eigenvalues of $B$.

SVDIFP is a recent extension to the inverse free preconditioned 
  Krylov subspace method, \cite{golub2002inverse}, for the singular value problem 
  \cite{liang2014computing}. 
The implementation includes the robust incomplete factorization (RIF) 
  \cite{benzi2003robust} for the normal equations matrix, but other preconditioners can 
  also be used.
To circumvent the intrinsic difficulties of filtering out the zero
  eigenvalues of $B$, the method works with the normal 
  equations matrix $C$, but computes directly the smallest singular values 
  of $A$ and not the eigenvalues of $C$.
Thus good numerical accuracy can be achieved but, 
  as we show later, at the expense of efficiency.
Moreover, the design of SVDIFP is based on restarting with a single vector, 
  which is not effective when seeking more than one singular values. 

In this paper we present a preconditioned hybrid two-stage method, PHSVDS, that achieves both efficiency and accuracy for both largest and smallest singular values under limited memory.  
In the first stage, the proposed method PHSVDS solves an extreme eigenvalue problem on $C$ up to the user required accuracy or up to the accuracy achievable by the normal equations. If further accuracy is required, PHSVDS switches to a second stage where it utilizes the eigenvectors and eigenvalues from $C$ as initial guesses to a Jacobi-Davidson method on $B$, which has been enhanced by a refined projection method. The appropriate choices for tolerances, transitions, selection of target shifts, and initial guesses are handled automatically by the method. We also discuss how to precondition PHSVDS and to cope with possible issues that can arise. 
Our extensive numerical experiments show that PHSVDS,
  implemented on top of the eigensolver PRIMME \cite{stathopoulos2010primme}, 
can be considerably more efficient than all other methods when computing a few of the smallest singular triplets, even without a  preconditioner. With a good preconditioner, the PHSVDS method can be much more efficient and robust than the JDSVD and SVDIFP methods.

In Section 2 we motivate the two stage SVD method based on the convergence of
  Krylov methods to the smallest magnitude eigenvalue of $B$ and $C$. 
In Section 3, we develop the components of the two stage method.
In Section 4, we describe how to precondition PHSVDS, and how to dynamically 
  inspect the quality of preconditioning at the two different stages. 
In Section 5, we present extensive experiments that corroborate our conclusions. 

We denote by $\|.\|$ the 2-norm of a vector or a matrix, 
by $A^T$ the transpose of $A$, by $I$ the identity matrix,
$\kappa(A) = \frac{\sigma_n}{\sigma_1}$, and
by $K_k(A,v) = span\{v, Av, \ldots, A^{k-1}v\}$ the k-dimensional Krylov 
subspace generated by $A$ and the initial vector $v$.

\section{Motivation for a two stage strategy}
We first need to understand whether an eigensolver on $C$ or on $B$ 
  is preferable in terms of convergence and accuracy 
  to iterative solvers that solve the SVD directly.
To address this, we introduce the basic SVD iterative methods 
  and study the asymptotic convergence and the quality of the 
  Krylov subspaces built by different methods.
We conclude that an appropriate choice of eigensolvers and eigenvalue
  problem yields methods that are faster and equally accurate as the best 
  SVD methods, especially in the presence of limited memory.

\subsection{The LBD, JDSVD, and SVDIFP methods}
The LBD method, \cite{golub1965calculating,golub1981block}, starts with unit vectors $p_1$ 
  and $q_1$ and after $k$ steps produces the following decomposition as 
  a partial Lanczos bidiagonalization of $A$:
\begin{equation}
\begin{aligned}
& AP_k = Q_kB_k, \\
& A^TQ_k = P_kB_k^T + r_ke_k^T,
\end{aligned}
\end{equation}
where the $r_k$ is the residual vector at $k$-th step, $e_k$ the 
  $k-$th orthocanonical vector, 
\[ B_k = \left( \begin{array}{cccc}
\alpha_1 & \beta_1 &  &  \\
 & \alpha_2 & \ddots  & \\
&  & \ddots & \beta_{k-1}  \\
&  &  & \alpha_k \end{array} \right)
= Q_k^T A P_k, \] 
and $Q_k$ and $P_k$ are orthonormal bases of the Krylov subspaces 
  $K_k(AA^T, q_1)$, and $K_k(A^TA, p_1)$ respectively.
With properly chosen starting vectors, LBD produces 
  mathematically the same space as the symmetric Lanczos method 
  on $B$ or $C$ \cite{jia2010refined,kokiopoulou2004computing}. 

To approximate singular triplets of $A$, LBD solves the small 
  singular value problem on $B_k$, and uses the corresponding 
  Ritz approximations from $Q_k$ and $P_k$ as left and right 
  singular vectors. 
To address the rapid loss of orthogonality of the columns of $Q_k$ 
  and $P_k$ in finite precision arithmetic, 
  full \cite{golub1981block}, 
  partial \cite{larsen2001combining}, 
  or one-sided reorthogonalization \cite{baglama2005augmented, jia2010refined}
  strategies have been applied to variants of LBD.
With appropriate implementation, these can result in a backward 
  stable algorithm for both singular values and vectors 
  \cite{barlow2013reorthogonalization,jia2010refined}.
Because, all these solutions become expensive when $k$ is large,
  restarted LBD versions have been studied 
  \cite{kokiopoulou2004computing,baglama2005augmented,jia2003implicitly,larsen2001combining}.
The goal is twofold: restart with sufficient subspace information 
  to maintain a good convergence, and identify the appropriate 
  Ritz information to restart with.
The former problem is tackled with implicit or thick restarting 
  \cite{stathopoulos1998dynamic}. 
The latter problem is tackled with combinations of harmonic and refined
  projection methods.
For example, IRLBA \cite{baglama2006restarted} uses a thick restarted
   block LBD with harmonic projection, 
  while IRRHLB \cite{jia2010refined} first computes harmonic Ritz vectors,
  and then uses their Rayleigh quotients in a refined projection to extract 
  refined Ritz vectors from $P_k$ and $Q_k$.

The JDSVD method \cite{hochstenbach2001jacobi} extends the Jacobi-Davidson method and its 
  correction equation for singular value problems by exploiting the 
  special structure of the augmented matrix $B$.
Similarly to LBD, JDSVD computes singular values, not eigenvalues,
  of the projection matrix, and the left and right singular vectors 
  from separate spaces.
Because good quality approximations are important not only for 
  restarting but also in the correction equation, various projection 
  methods can benefit JDSVD.
We introduce only the standard choice where the test and search space are the same.

Let $U$ and $V$ be the bases of the left and right search spaces.
Computing a singular triplet $(\theta,c,d)$ of $H = U^T A V$ yields 
  $(\theta,Uc,Vd)$ as the 
  Ritz approximation of a corresponding singular triplet of $A$. 
Alternatively, the $u=Uc$ and $v=Vc$ can be computed as harmonic or refined 
  singular triplets.
Then JDSVD obtains corrections $s$ and $t$ for $u$ and $v$ by solving 
  (approximately) the following correction equation:
\begin{equation}
\left( \begin{array}{cc}
P_u& 0 \\
0 & P_v  \end{array} \right)
\left( \begin{array}{cc}
-\theta I_m & A \\
A^T & -\theta I_n  \end{array} \right)
\left( \begin{array}{cc}
P_u & 0 \\
0 & P_v  \end{array} \right)
\left( \begin{array}{c}
s \\
t  \end{array} \right)  =
\left( \begin{array}{c}
A v - \theta u \\
A^T u - \theta v \end{array} \right)
\label{JDSVD_correction}
\end{equation}
where $P_v = I_n - vv^T, P_u = I_m-uu^T$. 
The left and right corrections $s,t$ are then orthogonalized 
  against and appended to $U$ and $V$ respectively. 
JDSVD uses thick restarting \cite{stathopoulos1998dynamic, wu2000thick} but retains also
  Ritz vectors from the previous iteration, similarly to the 
  locally optimal Conjugate Gradient recurrence \cite{stathopoulos2007nearlyI}. 
Most importantly, the JDSVD method can take advantage of preconditioning
  when solving (\ref{JDSVD_correction}).
  
The SVDIFP method \cite{liang2014computing} extends the EIGIFP method \cite{golub2002inverse}.
Given an approximation $(x_i, \rho_i)$ at the $i$-th step of the outer method, 
  it builds a Krylov space ${\cal V} = K_k(M(C - \rho_i I), x_i)$, where
  $M$ is a preconditioner for $C$.
To avoid the numerical problems of projecting on $C$, SVDIFP computes 
  the smallest singular values of $A{\cal V}$, by using a two sided projection
  similarly to the LBD.
Because the method focuses only the right singular vector, the left 
  singular vectors can be quite inaccurate.

\subsection{Asymptotic convergence of Krylov methods on $C$ and $B$}
\label{sec:asymptotic}
When seeking largest singular values, it is accepted that Krylov
methods on $C$ are faster than on $B$ \cite{hochstenbach2001jacobi,kokiopoulou2004computing,liang2014computing,berry1992large}.  
The argument is straightforward.

\begin{theorem}
\label{largest_convergence}
Let
$\gamma_B = \frac{\sigma_{n} - \sigma_{n-1}}{\sigma_{n-1} + \sigma_{n}}$ and 
$\gamma_C = \frac{\sigma_{n}^2 - \sigma_{n-1}^2}{\sigma_{n-1}^2 - \sigma_{1}^2}$
be the gap ratios of the largest eigenvalue of matrices $B$ and $C$, respectively.
Then, for the largest eigenvalue, the asymptotic convergence of 
  Lanczos on $C$ is $2$ times faster than Lanczos on $B$.
\end{theorem}
\begin{proof}
The asymptotic convergence rate is the square root of the gap ratio. 
Then:
$$
\gamma_C = \frac{(\sigma_{n} - \sigma_{n-1})(\sigma_{n} + \sigma_{n-1})^2}
	{(\sigma_{n}+\sigma_{n-1})(\sigma_{n-1}^2 - \sigma_{1}^2)} 
	= \gamma_B \frac{(\sigma_{n} + \sigma_{n-1})^2}{(\sigma_{n-1}^2 - \sigma_{1}^2)} 
	> \gamma_B \frac{4\sigma_{n-1}^2}{(\sigma_{n-1}^2 - \sigma_{1}^2)}
	= \frac{4 \gamma_B}{1-(\frac{\sigma_1}{\sigma_{n-1}})^2}.
$$
Therefore, for $\sigma_1\approx 0$, the asymptotic convergence rate 
  $\sqrt{\gamma_C}> 2\sqrt{\gamma_B}$.  In the less interesting case
  $\sigma_1 \rightarrow \sigma_{n-1}$,
  Lanczos on $C$ is arbitrarily faster than on $B$.
\end{proof}

For smallest singular values the literature is less clear, although methods that 
  work on $C$ have been avoided for numerical reasons. 
In previous experiments we have observed much faster convergence with approaches 
  on $C$ than on $B$ \cite{WM-CS-2014-03}. 
To obtain some intuition, we perform a basic asymptotic convergence analysis 
  of Krylov methods working on $C$ or on $B$ trying to compute the smallest 
  magnitude eigenvalue. 

\begin{lemma}
\label{LemmaforConvergenceOAAO}
Let the union of two intervals: $K = [-a, -b] \cup [c, d]$, $-b<0<c$, 
  and $p_k(x)$ the optimal degree $k$ polynomial that is as small as possible on $K$ 
  and $p_k(0) = 1$. Let $\epsilon_k = \max_{x\in K} |p_k(x)|$, 
  and $\rho = \lim_{k\rightarrow \infty} \epsilon_k^{1/k}$. Then asymptotically:
$$
\rho \simeq 1 - \sqrt{\frac{bc}{da}}.
$$
\end{lemma}
\begin{proof}
This is an application of Theorem 5 in \cite{shen2001potential}. 
\end{proof}

This $\rho$ translates to an upper bound for the asymptotic convergence rate 
  of any Krylov solver applied to an indefinite matrix whose spectrum lies in 
  the interval $K$.
Thus it can also be used for the convergence rate to the
  smallest positive eigenvalue of the augmented matrix $B$.
Assume $\sigma_1$ is a simple eigenvalue of $B$ and thus $\sigma_1^2$
  is a simple eigenvalue of $C$. Define its gap ratio in $C$ as,
  $\gamma = \frac{\sigma_{2}^2 - \sigma_{1}^2}{\sigma_{n}^2 - \sigma_{2}^2}$, 
  and assume $\gamma \ll 1$.

\begin{theorem}
\label{Convergence_indefinite} 
Consider the spectrum of the matrix $B-\sigma_1 I$, which lies 
  (except for the zero eigenvalue) in the two intervals: 
$K = [-\sigma_{n}-\sigma_{1}, -2\sigma_{1}] \cup 
     [ \sigma_{2}-\sigma_{1}, \sigma_{n}-\sigma_{1}]$.
The asymptotic convergence rate for any Krylov solver that finds 
  $\sigma_1$ is bounded by:
$$
\rho = 1 - \sqrt{\gamma \ \frac{2\sigma_{1}}{\sigma_{2}+\sigma_{1}} \
   	\frac{\sigma_n^2-\sigma_2^2}{\sigma_n^2-\sigma_1^2}}.
$$
\end{theorem}
\begin{proof}
Clearly, the optimal polynomial $p_k(x)$ of Lemma \ref{LemmaforConvergenceOAAO} 
  is the best polynomial for finding $\sigma_1$.
Applying the Lemma for the specific bounds for this interval we get:
\begin{equation*}
  \frac{bc}{ad} 
= \frac{2\sigma_{1}(\sigma_{2}-\sigma_{1})}{(\sigma_{n}+\sigma_{1})(\sigma_{n}-\sigma_{1})} 
= \frac{2\sigma_{1}}{\sigma_{2}+\sigma_{1}} \ \frac{\sigma_{2}^2-\sigma_{1}^2}{\sigma_{n}^2-\sigma_{1}^2}
= \frac{2\sigma_1}{\sigma_2+\sigma_1} \ \frac{\sigma_2^2-\sigma_1^2}{\sigma_n^2-\sigma_2^2} \
   \frac{\sigma_n^2-\sigma_2^2}{\sigma_n^2-\sigma_1^2}.
\end{equation*}
\end{proof}

\begin{lemma}
\label{ConvergenceATA}
The bound of the asymptotic convergence rate to $\sigma_1^2$ of Lanczos on $C$ 
  is approximately:
$ q = 1 - 2\sqrt{\gamma}$.
\end{lemma}
\begin{proof}
The bound on the rate of convergence of Lanczos for $\sigma_1^2$ is approximated as
   $e^{-2\sqrt{\gamma}}$ \cite[p. 280]{parlett1980symmetric}.
Taking the first order approximation from Taylor series around 0, we obtain 
   $e^{-2\sqrt{\gamma}} = 1 - 2\sqrt{\gamma} + O(\gamma)$.
\end{proof}

\begin{theorem}
\label{Speed_Ratio} 
A Krylov method on $C$ that computes $\sigma_1^2$ has always faster 
  asymptotic convergence rate than a Krylov method on $B$ that finds $\sigma_1$,
  by a factor of
\begin{equation} 
\tau = \frac{1 - \sqrt{\gamma} \sqrt{ \frac{2\sigma_{1}}{\sigma_{2}+\sigma_{1}} \
	     \frac{\sigma_n^2-\sigma_2^2}{\sigma_n^2-\sigma_1^2} }
            }
            {1 - 2\sqrt{\gamma}}.
\end{equation}
\end{theorem}
\begin{proof}
For the method on $C$ to be faster it must hold $\tau > 1$ or 
  $ \frac{2\sigma_{1}}{\sigma_{2}+\sigma_{1}} \
    \frac{\sigma_n^2-\sigma_2^2}{\sigma_n^2-\sigma_1^2} < 4$.
Basic manipulations lead to the condition 
  $ (4-2 \frac{\sigma_n^2-\sigma_2^2}{\sigma_n^2-\sigma_1^2}) \sigma_1 > = -4 \sigma_2$.
Since $\frac{\sigma_n^2-\sigma_2^2}{\sigma_n^2-\sigma_1^2} < 1$ and all $\sigma_i > 0$, 
  the above condition always holds. 
\end{proof}

First, we observe that if $\sigma_1$ is very close to 0, 
  the normal equations approach becomes arbitrarily faster than the augmented one,
  as long as $\sigma_2$ remains bounded away from 0. 
Second, it is not hard to see that $\tau = 1 + O(\sqrt{\sigma_2-\sigma_1})$, 
  which means that the two approaches become similar with highly clustered eigenvalues.
In that case, however, using a block method would increase the gap ratios and the
  gains from the approach on $C$ would be larger again.

Most importantly, the above asymptotic convergence rates reflect optimal 
  methods applied to $C$ and $B$ and an extraction of the best 
  information from the subspaces.
In practice, memory and computational requirements necessitate 
  the restarting of iterative methods, which results in significant 
  convergence slow down.
For extremal eigenvalues, combinations of thick restarting with the 
  locally optimal conjugate gradient directions have been shown to 
  almost fully restore the convergence of the unrestarted Lanczos method.
The GD+k and extensions to LOBPCG are such nearly-optimal methods
  \cite{stathopoulos1998dynamic,stathopoulos2007nearlyI}.
For interior eigenvalues, practical Krylov methods not only have a 
  hard time achieving this convergence, but also have problems
  extracting the best eigenvectors from the subspace.
Therefore, we expect in practice the normal equations to be significantly 
  faster than any approach based on $B$.
  
\subsection{Comparison of subspaces from Lanczos, LBD and JDSVD}
\label{sec:subspaces}
We extend the discussion on Lanczos to include two native SVD methods,
  and infer the relative differences between their convergence
  by studying the subspace they build.
A higher dimensional Krylov subspace implies faster convergence, assuming
  eigenvector approximations can be extracted effectively from the subspace.
We compare LBD, JDSVD, and Lanczos (or equivalently unpreconditioned GD) 
  on $C$ and on $B$.

Suppose $u_1, v_1$ are left and right initial guesses.
After $k$ iterations ($2k$ matvecs), Lanczos working on the normal equations matrix $C$ builds:
\begin{equation}
V_k = K_k(A^TA, v_1).
\end{equation}
The LBD method builds both left and right Krylov spaces \cite{baglama2005augmented}:
\begin{equation}
U_k = K_k(AA^T, A v_1), \qquad V_k = K_k(A^TA, v_1).
\end{equation}
The JDSVD method also builds two subspaces, each being a direct sum of two 
  Krylov spaces of half the dimension \cite{hochstenbach2001jacobi}: 
\begin{equation}
U_k = K_{\frac{k}{2}}(AA^T, u_1) \oplus K_{\frac{k}{2}}(AA^T, A v_1), 
\quad V_k = K_{\frac{k}{2}}(A^TA, v_1) \oplus K_{\frac{k}{2}}(A^TA, A^T u_1)
\end{equation}
Lanczos working on $B$ builds $K_k(B,[v_1;u_1])$ which does not correspond 
  exactly to the spaces above in general. 
In the special case of $u_1 = 0$, the subspace is given below:
\begin{equation}
\left( \begin{array}{c} U_k \\ V_k \end{array} \right) = 
\left( \begin{array}{c} 0  \\
                        K_{\frac{k}{2}}(A^TA, v_1) \end{array} \right)
\oplus
\left( \begin{array}{c} K_{\frac{k}{2}}(AA^T, A v_1) \\
                        0                          \end{array} \right).
\end{equation}

Clearly, Lanczos working on $C$ and LBD build the same Krylov subspace 
  for right singular vectors. 
The LBD method also builds the Krylov subspace for left singular vectors, and while 
  that helps generate the bidiagonal projection, it does not improve convergence 
  over Lanczos on $C$.
On the other hand, Lanczos on $B$ and JDSVD build a $k$ vector 
 subspace, but this comes from a direct sum of Krylov spaces of $k/2$ dimension.
Thus, they are expected to take twice the number of iterations of LBD in the 
  worst case.
The JDSVD subspace can be richer than that of Lanczos on $B$ because JDSVD
  handles the left and right search spaces independently for arbitrary initial 
  guesses.

\begin{figure}
  \centering
  \subfigure[unrestarted]{\includegraphics[width=0.49\textwidth]{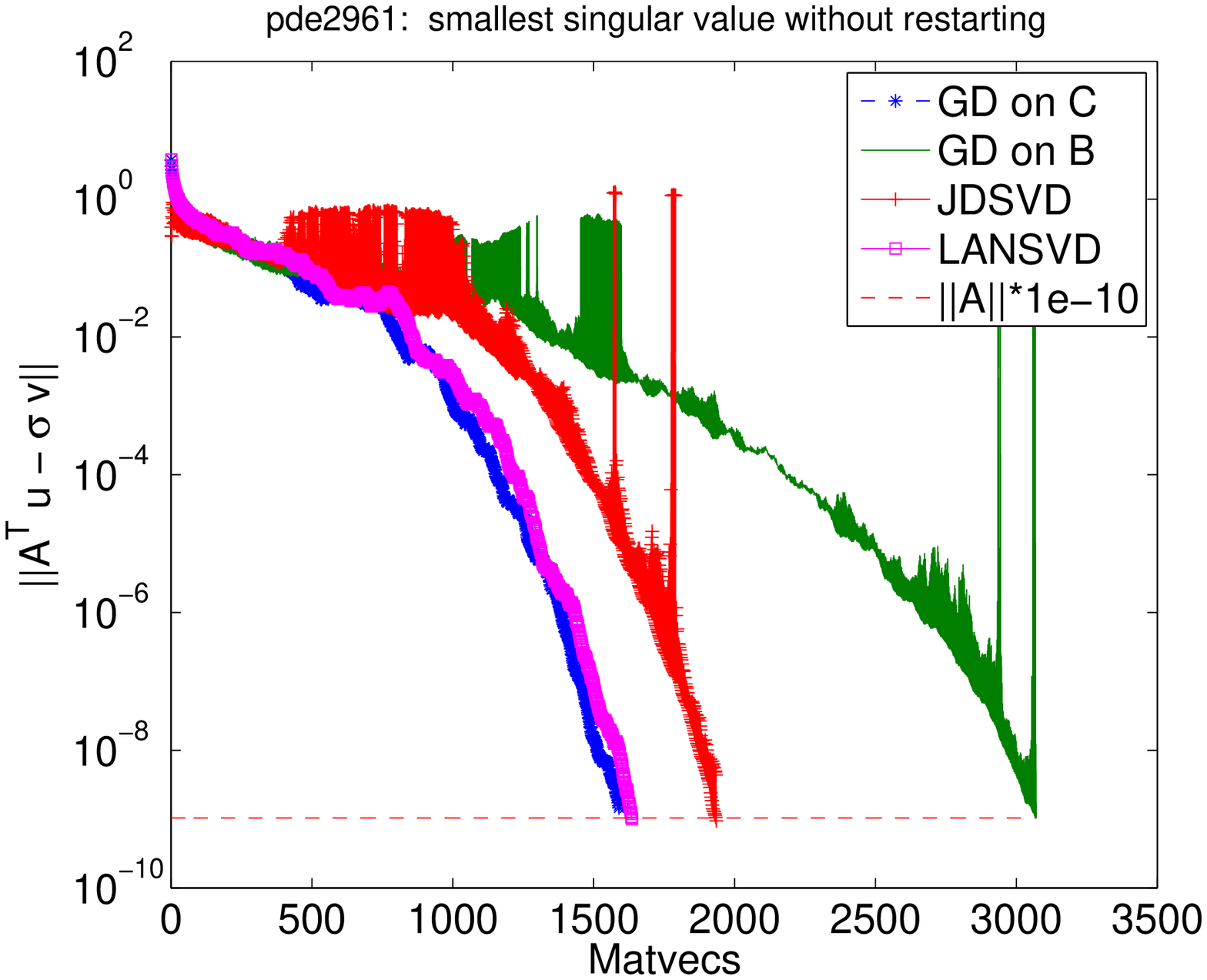}\label{fig:motivation_unrestarted}}                
 \subfigure[restarted]{\includegraphics[width=0.49\textwidth]{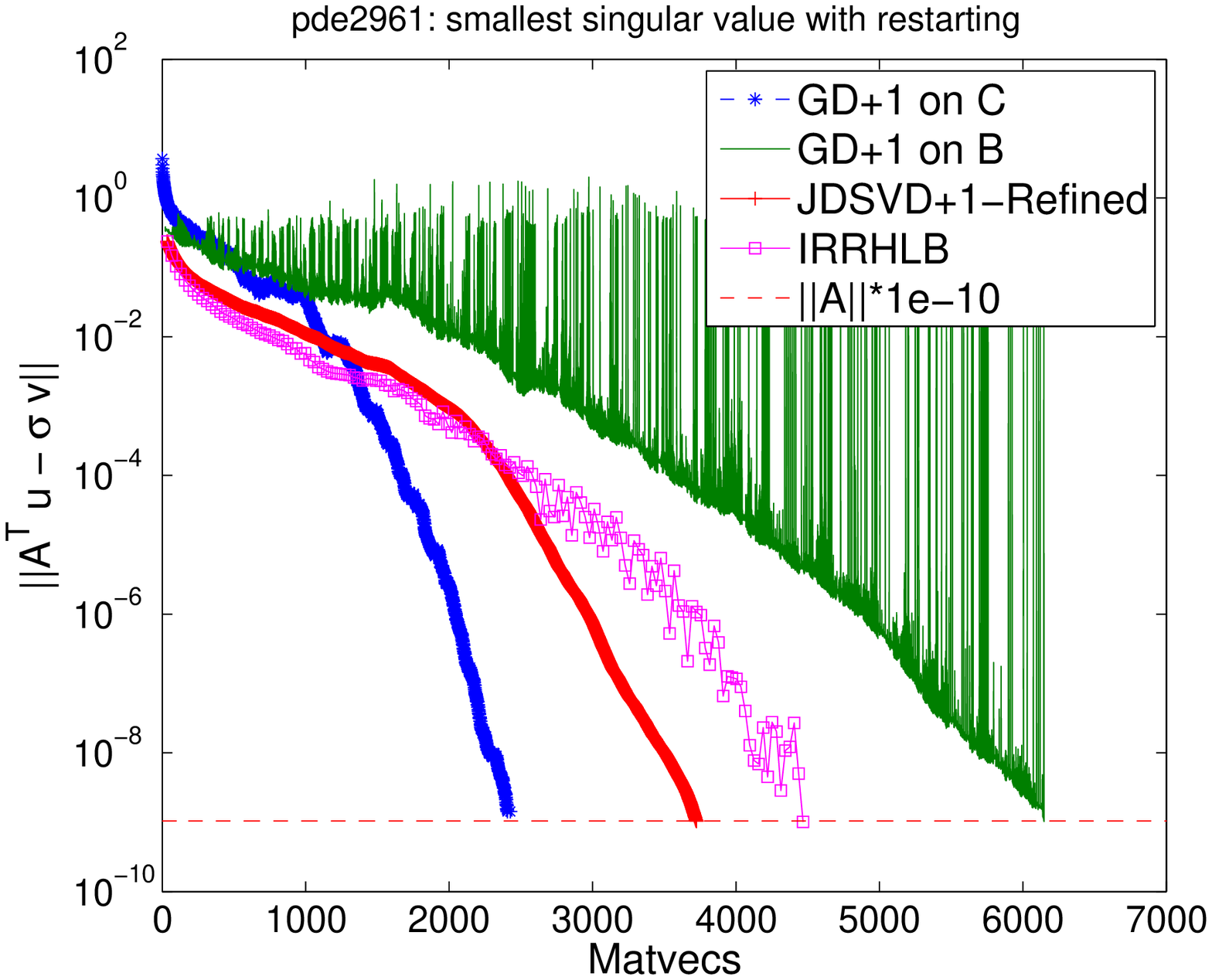}
 \label{fig:motivation_restarted}}   
  \caption{Comparing convergence speed of eigenmethods on $C$, $B$, 
LBD, and JDSVD in both unrestarted and restarted case for matrix pde2961. 
LANSVD implements LBD without restarting \cite{larsen2001combining} while IRRHLB 
  is currently the most advanced LBD method with implicit restarting 
  \cite{jia2010refined}.}
  \label{fig:motivation}
\end{figure}

Figure \ref{fig:motivation_unrestarted} demonstrates the relative convergence behavior of these 
  unrestarted methods seeking the smallest singular value of a sample matrix. 
Only the outer iteration of JDSVD is used (inner iterations = 0).
The results agree with the above analysis. The convergence speed of LBD is the same as GD on $C$. 
JDSVD is slower than LBD or GD but faster than GD on $B$ which is about twice as slow as LBD. 

These methods will inevitably be used with restarting.
Because LBD, JDSVD, and GD on $B$ extract interior spectral information from 
  the subspaces, critical directions may be dropped during restarting,
  causing significant convergence slow downs and irregular behavior.
The use of harmonic or refined Ritz projections during restart help 
  ameliorate this problem up to a point. 
However, the problem is still an interior one.
In contrast, GD+k on $C$ should see a far smaller effect on its convergence.

Figure \ref{fig:motivation_restarted} reflects the above and the advantage
  of solving an extreme eigenproblem with GD+k.
The only disadvantage is the limited accuracy because of the squared 
  conditioning of $C$.
Thus, a natural idea is to apply another phase to refine the accuracy 
  until user requirements are satisfied. 
Instead of iterative refinement, we claim that a second stage eigensolver on $B$ is more efficient.

\subsection{Necessary eigensolver features}
\label{subsection: Necessary eigensolver features}
Our goal is to develop a method that solves large, sparse 
  singular value problems with unprecedented efficiency, robustness, and 
  accuracy. 
This requires eigensolvers with a certain set of features.
First, the eigensolver should be able to use preconditioning because very slow 
  convergence is a limiting factor for seeking smallest singular triplets.
Second, large problem size suggests the use of advanced restarting techniques 
  so that limiting memory does not impede convergence.
Third, the eigensolver of the second stage should be able to exploit
  the good quality of the several singular vectors and singular values 
  computed in the first stage.
Fourth, the eigensolver on $B$ can benefit from refined or Harmonic 
  Ritz procedures for computing interior eigenvalues.

The state-of-the-art package PRIMME 
  (PReconditioned Iterative MultiMethod Eigensolver) \cite{stathopoulos2010primme}
  implements the GD+k and JDQMR methods that satisfy most of the 
  above requirements and provides a host of additional features.
With a few modifications we have developed our method on top of PRIMME.
However, appropriate eigensolvers from other packages could also be used.

\section{Developing the two stage strategy}
We develop PHSVDS, a two-stage SVD meta-method that first gets a fast solution of the eigenvalue problem 
  on $C$ to the best accuracy possible, and then resolves the remaining 
  accuracy with an eigensolver on $B$.
We discuss and automate issues of accuracy, convergence tolerance, initial guesses, and interior eigenvalues of $B$.
  
\subsection{The first stage of PHSVDS}
Although an eigensolver on $C$ can be much faster than other methods, 
  the residual norms of the eigenvalues involve $\|C\|=\|A\|^2$. 
Thus achieving the required numerical accuracy may not be possible. 

Let $(\sigma, u, v)$ be a targeted singular triplet of $A$ and 
  $(\tilde\sigma^2, \tilde v)$ the approximating Ritz pair from an
  eigenmethod working on $C$. 
Using the approximation $\tilde u = A\tilde v /\tilde\sigma $, 
  we can write the following four residuals:
\begin{equation}
\label{eq: different residuals}
\begin{array}{llll}
 r_v =  A   \tilde v - \tilde \sigma \tilde u, &  
 r_u =  A^T \tilde u - \tilde \sigma \tilde v, &
 r_C =  C \tilde v - \tilde \sigma^2 \tilde v, & 
 r_B =  B \left[\begin{array}{c}\tilde v\\ \tilde u\end{array}\right] - \tilde \sigma 
	   \left[\begin{array}{c}\tilde v\\ \tilde u\end{array}\right].
\end{array}
\end{equation}
Typically a singular triplet is considered converged when $\|r_v\|$ and 
$\|r_u\|$ are less than a given tolerance. Since our eigenvalue methods
 work on $C$ and $B$ we need to relate the above quantities. First, it is 
 easy to see that $r_C = A^T( A\tilde v) - \tilde \sigma^2 \tilde v 
 = \tilde \sigma A^T \tilde u - \tilde \sigma^2 \tilde v  = \tilde \sigma r_u$.
To relate to the norm of the residual of the second stage note that
$
\|r_{B}\|^2  = (\|r_v\|^2 + \|r_u\|^2) / (\|\tilde v\|^2 + \|\tilde u\|^2).
$
If the Ritz vector is normalized, $\|\tilde v\| = 1$, we also obtain
$\|\tilde u\| = \|A\tilde v/\tilde \sigma\|=1$ and $r_v = 0$.
Bringing it all together (see also  \cite{WM-CS-2014-03, liang2014computing}),
\begin{equation}
\|r_u\| = \frac{\|r_{C}\|}{\tilde \sigma}  = \|r_B\| \sqrt{2}.
\end{equation} 

Given a user requirement $\|r_u\| < \delta ~\|A\|$, the normal equations 
and the augmented methods should be stopped when 
  $\|r_{C}\| < \delta ~\tilde \sigma ~\|A\|$ and
  $\|r_{B}\| < \delta ~\|A\|/\sqrt{2}$ respectively. 
A common stopping criterion for eigensolvers is $\|r_C\| < \delta_C \|C\|$, 
  so we must provide $\delta_C = \delta ~\tilde \sigma / \|A\|$.
In floating point arithmetic this may not be achievable since $\|r_{C}\|$ 
  can only be guaranteed to achieve 
  $O(\|C\| \epsilon_{mach})$ \cite{parlett1980symmetric}. 
Thus, we use 
\begin{equation}
\delta_{C}  = \max \left(\delta ~\tilde\sigma/{\|A\|}, \epsilon_{mach} \right)
\end{equation}
as the criterion for the normal equations.

First, note that for the largest $\sigma_n$, $\delta_C = \delta$ and thus
  full residual accuracy is achievable with the normal equations.
Since $\sigma\approx \tilde\sigma$, based on the Bauer-Fike bound,
$ 
|\sigma^2 - \tilde\sigma^2 | \approx 
|\sigma - \tilde\sigma | (2\tilde\sigma) \leq \|r_C\| < 
	\delta_C \|A\|^2  = \tilde\sigma \delta \|A\|$ and thus
$|\sigma - \tilde\sigma | \leq \delta \|A\|/2$
so the singular values are as accurate as can be expected.

This does not hold for smaller, and in particular the smallest few, eigenvalues. 
Thus, if the user requires $\delta < {\|A\|} \epsilon_{mach}/ \tilde\sigma$, 
  PHSVDS first makes full use of the first stage and then switches to 
  the second stage working on $B$ to resolve the remaining accuracy of
  $O(\tilde\sigma/\|A\|) < \kappa(A)^{-1}$. 
For not too ill conditioned matrices, most of the time is then spent on 
  the more efficient first stage.

A second, more subtle issue involves the accuracy of the Ritz vectors 
  from $C$ which are used as initial guesses to $B$. 
We have observed that even though their residual norms are below
  the desired tolerance, the convergence of the interior eigenvalues
  in $B$ is sometimes (but not often) irregular, with long plateaus, 
  and might not be able to reach machine precision.
This occurs when the eigenvalues are highly clustered.
On the other hand, it does not occur when only one eigenvalue is
  sought, which implies that it has to do with the sensitivity of interior 
  eigenvalues to the nearby eigenvectors that we pass as initial guesses
  \cite{Morgan1998Harmonic}.
Therefore, before we start stage two, we perform a complete Rayleigh 
  Ritz procedure with the converged eigenvectors of $C$.
Providing the new Ritz vectors as initial guesses completely cures
  this problem. 

To understand the problem as well as the solution, consider the 
  decomposition of the smallest Ritz vector $\tilde u_1$ on the exact 
  eigenvectors of $C$,
$\tilde u_1 = c_1 u_1 + \sum_{i=2}^n c_i u_i$.
On exit from the first stage, its residual satisfies 
 $\|r_{\tilde u_1}\|<\|C\|\delta_C$, and from Bauer-Fike it also holds,
 $\sqrt{1-c_1^2} < \|C\|\delta_C$ and $c_i \leq \|C\| \delta_{C}$. 
Therefore, if we omit second and higher order terms, the Rayleigh quotient 
  and the residual of $\tilde u_1$ can be written as:
\begin{eqnarray}
\mu = \frac{\tilde u_1^T A \tilde u_1}{\tilde u_1^T \tilde u_1} 
    = \lambda_1 + \sum\limits_{i=2}^n (\frac{c_i}{c_1})^2 \lambda_i, \qquad
r_{\tilde u_1} = A \tilde u_1 - \mu \tilde u_1 
    \approx \sum\limits_{i=2}^n c_i (\lambda_i - \lambda_1)u_i.
\label{eq: mu/residual of smallest eigenvalue}
\end{eqnarray}
As a result of the convergence behavior of iterative methods, 
  the $c_i$ tend to be larger for nearby eigenpairs, and fall drastically 
  as $i$ increases.
Then, from (\ref{eq: mu/residual of smallest eigenvalue}), 
  the accuracy of $\mu$ is dominated by nearby $(c_i/c_1)^2$.
The post-processing Rayleigh-Ritz uses the $k$ nearby converged 
  Ritz vectors, recomputes the projection with less floating point errors,
  and rearranges the directions to produce smaller $c_i, i=2,\ldots,k$, 
  and thus better Ritz values.
Of course, the additional Rayleigh-Ritz cannot improve the residual
  norms without the incorporation of new information in the basis.
Even in floating point arithmetic, the improvements are minimal.
This is evident also in (\ref{eq: mu/residual of smallest eigenvalue})
  where $r_{\tilde u_1}$ depends on the $c_i$ linearly, so the effect of 
  improving the nearby $c_i$ is small.
Figure \ref{fig:angle} shows these effects on a matrix 
  that presented the original problem.

\begin{figure} 
  \centering
\begin{minipage}{0.49\textwidth}
  \includegraphics[width=\textwidth]{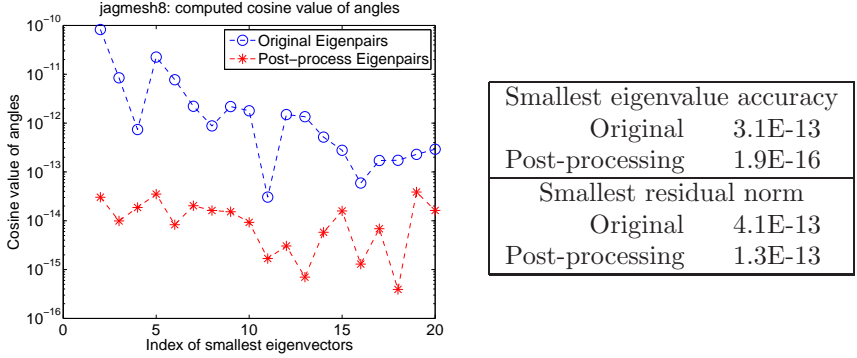}            
\end{minipage}
\begin{minipage}{0.42\textwidth}
    \mbox{\begin{tabular}{| r  c |}
    \hline
    \multicolumn{2}{|c|}{Smallest eigenvalue accuracy} \\
    Original & 3.1E-13 \\
    Post-processing & 1.9E-16 \\ \hline
    \multicolumn{2}{|c|}{Smallest residual norm} \\
    Original & 4.1E-13 \\ 
    Post-processing & 1.3E-13       \\ \hline
    \end{tabular}}
\end{minipage}

\caption{The cosine of angles ($c_i = \tilde u_1^T u_i, i=2,\ldots,19$) between 
  the smallest exact eigenvectors and the smallest Ritz vector before and after pre-processing.
  The table compares the accuracy of the Rayleigh quotient and the residual norm 
  for $\tilde u_1$ before and post-processing for matrix jagmesh8.}
  \label{fig:angle}
\end{figure}

In the second stage, the improvements on $c_i$ translate to a
  starting search space of better quality, and thus better 
  Ritz (or harmonic Ritz) pairs for the interior eigenvalue 
  problem. 
Algorithm \ref{alg: phsvds - 1st stage} summarizes the 
  interaction with the first stage eigensolver.

\begin{algorithm}
\caption{PHSVDS first stage}
\begin{algorithmic}[1]
	\STATE Call an eigesolver on $C$ to compute the required
		 ($\tilde \sigma_i^2, \tilde v_i)$ eigenpairs
	\STATE 
	Eigensolver convergence test uses dynamic tolerance 
	$\delta_{C}  = \max \left(\delta_{user} \frac{\tilde\sigma_i}{\tilde\sigma_n}, \epsilon_{mach} \right)$, where $\tilde\sigma_n$ and $\tilde\sigma_i$ are
	its current largest and targeted Ritz values
	\STATE If requested, perform Rayleigh-Ritz on the returned Ritz vector basis
\end{algorithmic}
\label{alg: phsvds - 1st stage}
\end{algorithm}

\subsection{The second stage of PHSVDS}
An eigensolver such as GD+k can be used also to compute
  interior eigenvalues of $B$ close to a given set of shifts.
However, the availability of accurate initial guesses and shifts from 
  the first phase suggest that an inner-outer iterative eigenmethod, 
  such as JDQMR, might be preferable.

We argue that solving an eigenvalue problem on matrix $B$ with an
  approximate eigenspace as initial guess is a better approach 
  than iterative refinement \cite{dongarra1983improving,berry1992large}.
First, with iterative refinement, eigenvectors are improved one by one 
  without any synergy from the nearby subspace information. 
In contrast, a subspace eigensolver 
  provides global convergence to all desired pairs.
Second, an inner-outer eigensolver such as JDQMR stops the inner linear 
  solver dynamically and near-optimally to avoid exiting too early (which 
  increases the number of outer iterations) or iterating too long 
  (which increases the number of inner iterations). 
We are not familiar with similar implementations for iterative refinement.
Third, iterative refinement for clustered interior eigenvalues may not be 
  able to converge to the desired high accuracy due to the lack of 
  proper deflation strategies \cite{arioli2014chebyshev}, both at the 
  linear solver and at the outer iteration. 
Naturally, a well designed eigensolver that employs locking and blocking 
  techniques is more robust to address these problems. 
Finally, we point out that the correction equation of the Jacobi-Davidson 
  method applied on $B^T$,
\begin{equation}
(I - w w^T) (B^T - \mu I) (I - ww^T) \tilde{t} = \tilde{\sigma} w - B^T w,
\end{equation}
  where $w^T = \left[ u^T v^T \right] $ 
is equivalent to the iterative refinement proposed in \cite{dongarra1983improving}
(\cite{hochstenbach2001jacobi}). 
Therefore, JDQMR enjoys the benefits of both eigensolvers and iterative refinement.

Using the eigenvector approximations 
  $\tilde v_i, \tilde u_i = A\tilde v_i /\tilde\sigma_i $ from the first stage, 
  we form initial guesses to insert in the search space of the JDQMR on $B$.
If the guesses are less than the minimum restart size, we fill the rest 
  of the positions with a Lanczos space from the first targeted eigenvector.
In addition, we provide the eigenvalue approximations from the first stage,
  which are typically very accurate because of Hermiticity, as shifts to JDQMR.
However, the problems with seeking interior eigenvalues are accentuated in 
  the maximally indefinite case of SVD problems. 
Spurious Ritz values can cause Ritz vectors to fail to converge \cite{stewart2001matrix} 
  or to have a detrimental effect during restart when major eigenvector components may be discarded and need to be recovered \cite{morgan1991computing, Morgan1998Harmonic, sleijpen2000jacobi}. 
We have addressed these problems as follows.
 
First, we observed that sometimes eigenvector approximations that are introduced initially in the search space but have not been targeted yet may degrade in quality or even be displaced. Thus, when an eigenvector converges and is locked out of the search space, we re-introduce the initial guess of the next vector to be targeted. This resulted in significant improvement in robustness and often in convergence speed.

Second, we introduced an efficient implementation of the refined projection 
  that minimizes the residual $\|BVy - \tilde \sigma Vy\|$ 
  for a given $\tilde \sigma$ over the search space $V$ 
  \cite{jia1997refined, stewart2001matrix}. 
Because the shifts $\tilde \sigma$ are accurate, a harmonic Ritz procedure is not necessary, and the refined one is expected to give the best approximation for the targeted eigenpair. 
Our refined projection is similar to the one in 
  \cite{hochstenbach2004harmonic,morgan1991computing} which produces refined Ritz vectors for all required eigenvectors (not just the closest to $\tilde \sigma$). 
Since $\tilde \sigma$ remains constant, there is no need to perform a QR factorization of $BV-\tilde \sigma V$ at every step.
Instead, as part of Gram-Schmidt, we update the factorization matrices
  $Q$ and $R$ with a new column.
A full QR factorization is only needed at restart. Then, following \cite{stewart2001matrix}, we compute the refined Ritz vectors by solving the small SVD problem with $R$, and replace the targeted Ritz value with the Rayleigh quotient of the first refined Ritz vector.

\begin{algorithm}
\caption{PHSVDS second stage: necessary enhancements in Jacobi-Davidson}
\begin{algorithmic}[1]
\STATE Initial shifts $\tilde \sigma_i$, initial vectors 
	$[\tilde v_{i}\ ; A\tilde v_{i}/\tilde \sigma_{i} ], i=1,...,k$,~
	$qr\_full = 1, ~j=s+1$
\STATE Build an orthonormal basis $V$ of $[K_{s-k}(B, \tilde v_1), \tilde v_i]$. Set $t$ as the Lanczos residual
\WHILE{all $k$ eigenvalues have not converged}
	\STATE Orthonormalize $t$ against $V$.
		Update $v_{j} = t, w_{j} = Bv_{j}, H_{:,j} = V^{T}w_{j}$
	\IF {$qr\_{full} == 1$}
		\STATE $W - \tilde \sigma_{1} V = QR, \ qr\_{full} = 0$
	\ELSE
		\STATE During orthogonalization update $QR$ by one-column
	\ENDIF
	\STATE Compute eigendecomposition $H = S\Theta S^{T}$ with $\theta_{j}$ ordered by closeness to $\tilde \sigma_1$ 
	\STATE Compute SVD decomposition of $R = U\Sigma S^{T}$, Rayleigh quotient $\theta_{1} = s_{1}^{T}Hs_{1}$
	\STATE If $(\sigma_{i},[v_i; u_i])$ converged, lock, and re-introduce 
		$[\tilde v_{i+1}\ ; A\tilde v_{i+1}/\tilde \sigma_{i+1}  ] $ into $V$
	\STATE If restarting, set $qr\_{full} = 1$
	\STATE Obtain the next vector $t = Prec(r)$ (e.g., by solving the correction equation)
\ENDWHILE
\end{algorithmic}
\label{alg: phsvds - 2nd stage}
\end{algorithm}

Solving the small SVD problem in each iteration for only the targeted shift 
  reduces the cost of the refined procedure considerably, making it similar 
  to the cost of computing the Ritz vectors. 
However, the quality of non-targeted refined Ritz vectors reduces with the 
  distance of their Rayleigh quotient from $\tilde \sigma$, so they may not 
  be as effective in a block algorithm.
Yet, these approximations have the desirable property of monotonic convergence as claimed in \cite{hochstenbach2004harmonic,morgan1991computing} and also observed in our experiments. This added robustness for JDQMR more than justifies the small additional cost.
Algorithm \ref{alg: phsvds - 2nd stage} summarizes the second 
  stage modifications in the context of JD.

\subsection{Outline of the implementation}
We have implemented the PHSVDS meta-method as a MATLAB function 
  on top of PRIMME.
This allowed us flexibility for algorithmically tuning the two stages
  and experimenting with various eigensolvers.
We first developed a MATLAB MEX interface for PRIMME,
  which exposes its full functionality to a broader class of users, 
  who can now take advantage of MATLAB's built-in matrix times block-of-vectors
  operators and preconditioners. 
Its user interface is similar to MATLAB {\tt eigs} and it is fully tunable.
Many of the enhancements, such as 
  the refined projection method or a user provided stopping criterion, 
  were implemented directly in PRIMME and will be part of 
its next release.
We are currently working on a native C implementation of PHSVDS in PRIMME.

PHSVDS expects an input matrix $A$, or a matrix function that 
  performs matrix-vector operations with $A$ and $A^T$, 
  or directly with $B$ and/or $C$.
Then, it sets up the matrix-vector functions and calls 
  Algorithm \ref{alg: phsvds - 1st stage}.
The returned approximations are provided as initial guesses to 
  Algorithm \ref{alg: phsvds - 2nd stage}.
One exception is when the singular value is extremely small or zero, 
  so the first stage yields no accuracy for $\tilde u_i$.
In this extreme case, it is better to choose $\tilde u_i$ as a 
  random vector.
For tiny and highly clustered singular values, eigensolvers with 
  locking and block features are preferable.

\section{Preconditioning in PHSVDS}

The shift-invert technique is sometimes thought of as a form of 
  preconditioning. 
If a factorization of $A, C$ or $B$ is possible, this
  is often the method of choice for highly clustered or indefinite 
  eigenproblems.
For smallest singular values, MATLAB {\tt svds} relies
  solely on shift-invert ARPACK \cite{lehoucq1998arpack} using the LU factorization of $B$.
PROPACK \cite{larsen1998lanczos,larsen2001combining} uses a QR factorization of $C$. 
However, for rectangular matrices {\tt svds} often converges to 
  the zero eigenvalues of $B$ rather than the smallest singular 
  value (see \cite{liang2014computing}).
Our method can also be used in shift-invert mode, assuming the 
  user provides the inverted operator as a matrix-vector.
For large matrices, however, preconditioners become a necessary 
  alternative.

JDSVD accepts a preconditioner for a square matrix $A$ or, if $A$ is 
  rectangular, leverages a preconditioner for $B - \tau I$ \cite{hochstenbach2001jacobi}. 
SVDIFP includes by default the robust incomplete factorization (RIF) 
  method \cite{benzi2003robust} to provide a preconditioner
  directly for $C - \tau I$ without forming $C$ \cite{liang2014computing}. 
The advantage is that it works seamlessly for both square and 
  rectangular matrices, but RIF may not be the best choice of
  preconditioner.

PHSVDS accepts any user-provided preconditioning operator that
  the underlying eigensolver allows.
In the most general form, any preconditioner directly for $C$ or $B$
  can be used.
When $M\approx A$ is available (e.g., the incomplete LU factorization 
  of a square matrix), PHSVDS forms $M^{-1}M^{-T}$ and 
  $[0 \ M^{-1};M^{-T} \ 0]$ as the preconditioning operators for the different stages.
Moreover, if a preconditioner such as RIF is given, $M\approx C^{-1}$, 
  we can build preconditioners for the second stage as $[0\ AM ; MA^T\ 0]$.
It is not clear in general how to form a preconditioner for $C$ from a 
  preconditioner of $B$.

\subsection{A dynamic two stage method with preconditioning}
The analysis in Sections \ref{sec:asymptotic} and \ref{sec:subspaces}
  holds for Krylov methods but it is less meaningful with preconditioners.
Clearly, if two different preconditioners are provided for $C$ and $B$
  their relative strengths are not known by PHSVDS.
But we have also noticed cases where the first stage benefits less 
  than the second stage when a less powerful preconditioner $M$ for $A$
  is used to form preconditioners for both $C$ and $B$.
If $M$ is ill-conditioned but its near-kernel space does not correspond
  well to that of $A$, it may work for $B$, but taking $M^TM$ 
  produces an unstable preconditioner for $C$ \cite{saad2001enhanced}. 
On the other hand, with a sufficiently good preconditioner, both methods 
  enjoy similar benefits on convergence.
If the relative strengths of the provided preconditioner are known,
  users can choose the two-stage approach or only one of the 
  stages (e.g., the second one). 
For the general case, we present a method that, based on runtime 
  measurements, switches dynamically between the normal equations and the 
  augmented approach to identify the most effective one for the
  given preconditioning. 
This is shown in Algorithm \ref{alg: phsvds with preconditioning}.
  
To estimate the convergence of the two approaches, we run a set 
  of initial tests alternating between running on $C$ and on $B$.
Because JDQMR relies on good initial guesses which are not available initially, 
  the dynamic algorithm uses only the GD+k method.
Once the algorithm decides on the approach, any eigenmethod that allows for
  preconditioning can be used.
Without loss of generality, we only consider GD+k for our 
  dynamic PHSVDS experiments.
The approximations obtained from one run are passed as initial guesses to the 
  next run.

We estimate the convergence rate by measuring the 
  average reduction per iteration of the residual norm.
To capture the convergence at different phases of 
  the iterative method, we must switch between the two approaches
  several times.
However, switching too frequently incurs a lot of overhead (rebuilding the initial 
  basis, performing extra Rayleigh-Ritz procedures, and possibly 
  convergence loss from restarting the search space).
Switching too infrequently may be wasteful when the preconditioner for $C$
  does not work well.
Thus, we control the maximum number of iterations for the next GD+k run, 
  $maxIter$.
This number is always larger than $initIter$ which is a reasonably 
  small number, i.e., 50. 
If the same approach is chosen in two successive runs, $maxIter$ doubles.
If the approach should be switched, $maxIter$ is reduced
  more aggressively for the next run (Step 12) 
  to avoid wasting too much time on the wrong approach.
If one eigenvalue converges in the initial tests 
  or at least two eigenvalues converge later, we stop the dynamic
  switching and choose the currently faster approach.
If the faster approach is the normal equations, a two-stage method
  might be necessary to get to full accuracy.
Although two or three switches typically suffice to distinguish between approaches,
  we also limit the number of switches.

\begin{algorithm}
\caption{Dynamic switching between stages for preconditioned PHSVDS}
\begin{algorithmic}[1]
	\STATE Set $initIter$, $maxSwitch$ 
	\STATE $numSwitch=numConverged=j=0, maxIter=initIter, Undecided = true$
	\STATE Run $initIter$ iterations of an eigensolver (such as GD+k) on $C$ and on $B$ and 
 		collect initial average convergence rate of both approaches.

	\WHILE{($numSwitch < maxSwitch$ and  $Undecided$)}
		\STATE Choose estimated faster approach ($C$ or $B$) for next call
		\IF{($numSwitch == 0$ {\bf and} $numConverged > 0)$
			{\bf or} $numConverged > 1$}
			\STATE undecided = false (Choose faster approach and no more switching)
		\ELSE
			\IF{Same approach is chosen again}
				\STATE $j=j+1;\ \ maxIter = initIter \ast 2^j$
	    		\ELSIF{Different approach is chosen}
	    			\STATE $j=\mbox{floor}(j/2);\ \ 
					maxIter = initIter \ast 2^j$
	    		\ENDIF
	    	\ENDIF
		\STATE $numSwitch=numSwitch+1$
		\STATE Call the eigensolver with $maxIter$ and current chosen approach
    	\ENDWHILE
	\IF{All desired singular triplets are found on $B$}
		    \STATE Return final singular triplets to users
	\ELSIF{All desired singular triplets are found on $C$}
		    \STATE Return resulting singular triplets to augmented approach
	\ELSIF{Faster approach is on $C$}
		    \STATE Proceed with the two-stage approach
	\ELSIF{Faster approach is on $B$}
		\STATE Continue only with the augmented approach
	\ENDIF

\end{algorithmic}
\label{alg: phsvds with preconditioning}
\end{algorithm}

\section{Numerical experiments}
Our first two experiments use diagonal matrices to demonstrate the principle 
  of the two stage method and that the method can compute 
  artificially clustered tiny singular values to full accuracy.
Then, we conduct an extensive set of experiments for finding the smallest 
  singular values of several matrices. 
Large singular values are also computed under the shift-invert setting.
The matrix set is chosen to overlap with those in other papers 
  in the literature.
We compare against several state-of-the-art SVD methods: 
	JDSVD \cite{hochstenbach2001jacobi,hochstenbach2004harmonic}, 
	SVDIFP \cite{liang2014computing}, 
	IRRHLB \cite{jia2010refined}, 
	IRLBA \cite{baglama2005augmented},  
	lansvd \cite{larsen2001combining},
    and MATLAB's {\tt svds}.
All methods are implemented in MATLAB.
First, we compute a few of the smallest singular triplets on both 
  square and rectangular matrices without a preconditioner. 
Then we show show the effect of the dynamic PHSVDS for 
  different quality of preconditioners, 
  and demonstrate that PHSVDS provides faster convergence over 
  other methods on these test matrices as well as on some 
  large scale problems.

All computations are carried out on a DELL dual socket with Intel Xeon processors 
  at 2.93GHz for a total of 16 cores and 50 GB of memory running the SUSE Linux 
  operating system.
We use MATLAB 2013a with machine precision $\epsilon = 2.2 \times 10^{-16}$ 
  and PRIMME is linked to the BLAS and LAPACK libraries available in MATLAB.
Our stopping criterion requires that the left and right residuals satisfy,
\begin{equation}
\sqrt{\|r_u\|^2 + \|r_v\|^2} < \|A\| \delta_{user}.
\end{equation}

For JDSVD we use the refined projection method as it performed best in 
  our experiments, which is also consistent with \cite{hochstenbach2001jacobi}. 
We choose the default for all parameters except for setting 'krylov = 0' 
  to avoid occasional convergence problems for smallest singular values.
For SVDIFP the maximum number of inner iterations can be chosen 
  as fixed or adaptive. 
We run with both choices and report the best result. 
Also, we modify slightly the code to use singular triplet residuals 
  as the stopping criterion for SVDIFP 
  instead of the default residual of the normal equations. 
For IRLBA and IRRHLB, we choose all default parameters as suggested in the code. 

All methods start with the same initial guess, {\tt ones(min(m,n),1)},
   except for matrix lshp3025 for which a random guess is necessary. 
We set the maximum number of restarts to 5000 for IRRHLB and IRLBA 
  and to 10000 for JDSVD and PHSVDS. 
Since SVDIFP can only set a maximum number of iterations for each 
  targeted singular triplet, we report that SVDIFP cannot converge 
  to all desired singular values if its overall number of matrix-vector 
  operations is larger than $(maxBasisSize - k)*5000$. 
For PHSVDS, we set maxBasisSize=35, minRestartSize=21 and experiment 
  with two $\delta$ tolerances, 1E-8 and 1E-14. 
For $\delta = $1E-8, PHSVDS does not need to enter the second stage 
  for any of our tests. For the first stage of PHSVDS we use the GD\_Olsen\_PlusK method. 
For the second stage, we run experiments with both GD\_PlusK and JDQMR. 
Since our implementation is mainly in C, we compare the number of 
  matrix-vector operations as the primary measurement of the performance.
However, we also report execution times which is relevant since matrix-vector, 
  preconditioner, and all BLAS/LAPACK operations are performed by 
  the MATLAB libraries.
  
Since the numbers of matrix-vector products with $A$ and $A^T$ are the same, the tables report as ``MV'' the number of products with $A$ only. ``Sec'' is the run time in seconds, and ``--'' means the method cannot converge to all desired singular values or that the code breaks down. 
Bargraphs report the ratio of matvecs and time of each method over 
  the PHSVDS method that uses the first stage only or JDQMR in stage two.
Ratio values are truncated to less than 10, and empty bar means 
  the same as ``--'' in the tables.

\subsection{PHSVDS for clustered tiny singular values}
We illustrate first how our two-stage method works in a seamless manner. 
We consider a diagonal matrix $A$ = diag([1:10,1000:100:1E6]) and the
  preconditioner $M$ = $A$ + diag(rand(1,10000)*1E4). 
In Figure \ref{fig: ATAtoOAAOswitch}, the green and black lines 
  show the convergence behaviors of GD+k on $B$ and $C$ respectively.
Indeed, the convergence on $B$ is very slow due to a highly indefinite problem 
  while the accuracy on $C$ stagnates when reaching its limit.
The two stage PHSVDS combines the benefits of the two methods, and 
  determines the smallest singular value efficiently and accurately as the magenta line shows.

\begin{figure}[bht]
\centering
\includegraphics[width=0.49\textwidth]{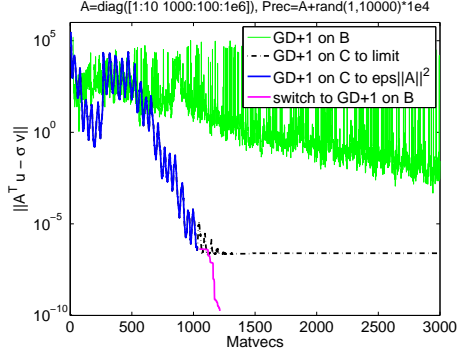}
\caption{Two stages of PHSVDS working seamlessly to find smallest singular values accurately.}
\label{fig: ATAtoOAAOswitch}
\end{figure}

Next we show how PHSVDS can determine several clustered tiny 
  singular values.
Consider the matrix $A$ = 
diag([1E-14, 1E-12, 1E-8:1E-8:4E-8, 1E-3:1E-3:1]) 
  with identity matrix as a preconditioner. 
Although locking may be able to determine clustered or multiple singular values, 
  we increase robustness by using a block size of two for the first stage only.
We set the user tolerance $\delta_{user}$ = 1E-15 to examine the 
  ultimate accuracy of PHSVDS. 
As shown in Table \ref{ta: ExpA}, PHSVDS is capable of computing 
  all the desired clustered, tiny singular values accurately.

\begin{table}[htbp]
\centering
\caption{Computation of 10 clustered smallest singular values by PHSVDS with block size 2.}
\label{ta: ExpA}
\small
\begin{center}
    \begin{tabular}{ |c|c|c|}
    \hline
 	     $\sigma_i$  & PHSVDS $\tilde \sigma_i$ & $\|r_B\|$ \\ \hline 
                 1E-14& 9.952750930151887E-15 & 4.9E-16 \\ \hline 
                 1E-12& 1.000014102726476E-12 & 8.9E-16 \\ \hline 
                 1E-8 & 1.000000003582006E-08 & 6.5E-16 \\ \hline 
                 2E-8 & 2.000000002073312E-08 & 9.2E-16 \\ \hline 
                 3E-8 & 3.000000000893043E-08 & 9.0E-16 \\ \hline 
    \end{tabular}
    \begin{tabular}{ |c|c|c|}
    \hline
 	     $\sigma_i$  & PHSVDS $\tilde \sigma_i$ & $\|r_B\|$ \\ \hline 
                 4E-8 & 4.000000001929591E-08 & 9.2E-16 \\ \hline 
                 1E-3 & 1.000000000000025E-03 & 9.7E-16 \\ \hline 
                 2E-3 & 1.999999999999956E-03 & 8.1E-16 \\ \hline 
                 3E-3 & 2.999999999999997E-03 & 9.1E-16 \\ \hline 
                 4E-3 & 3.999999999999958E-03 & 9.8E-16 \\ \hline 
    \end{tabular}
\end{center}
\end{table}

\subsection{Without preconditioning}
We compare two variants of PHSVDS with four methods, 
  JDSVD, SVDIFP, IRRHLB, and IRLBA on both square and rectangular 
  matrices without preconditioning. 
Since a good preconditioner is usually not easy to obtain for SVD problems, 
  it is important to examine the effectiveness of a method in this case. 
We compute $k=1, 10$ smallest singular triplets (see \cite{WM-CS-2014-06} 
  for a more detailed and extensive set of results).
In order to speed up the convergence of {\tt svds}, IRRHLB and IRLBA, 
  we compute $k+3$ eigenvalues when $k$ eigenvalues are required.
For PHSVDS, we found this is not necessary.

\begin{table}[htbp]
\centering
\caption{Properties of the test matrices. 
  $\gamma_{m}(k) = \min_{i=1}^k (gap(\sigma_i))$, and $gap(\sigma_i) = min_{j \neq i} |\sigma_{i} - \sigma_j|$.}
\label{ta: Square and rectangular matrices}
\small
\begin{center}
    \begin{tabular}{ |c|c|c|c|c|c|c|c|c|}
    \hline
    Matrix 			& pde2961 & dw2048 & fidap4 & jagmesh8 & wang3 & lshp3025 \\ \hline 
    order 			& 2961	 & 2048   & 1601   & 1141   & 26064  & 3025   \\ \hline
    nnz(A) 		    & 14585  & 10114  & 31837  & 7465   & 77168  & 120833  \\ \hline
    $\kappa(A)$ 		& 9.5E+2  & 5.3E+3  & 5.2E+3  & 5.9E+4  & 1.1E+4  & 2.2E+5   \\ \hline
    $\|A\|_{2}$ 		& 1.0E+1  & 1.0E+0  & 1.6E+0  & 6.8E+0  & 2.7E-1 & 7.0E+0  \\ \hline
    $\gamma_m(1)$  & 8.2E-3 & 2.6E-3 & 1.5E-3 & 1.7E-3 & 7.4E-5 & 1.8E-3  \\ \hline
    $\gamma_m(5)$  & 2.4E-3 & 2.9E-4 & 2.5E-4 & 4.8E-5 & 1.9E-5 & 1.8E-4  \\ \hline
    $\gamma_m(10)$ & 7.0E-4 & 1.6E-4 & 2.5E-4 & 4.8E-5 & 6.6E-6 & 2.2E-5  \\ \hline
    \end{tabular}
\end{center}

\small
\begin{center}
    \begin{tabular}{ |c|c|c|c|c|c|c|}
    \hline
    Matrix 			& well1850 & lp\_ganges & deter4 & plddb & ch  & lp\_bnl2  \\ \hline
    rows $m$: 			& 1850 & 1309 & 3235 & 3049  & 3700 & 2324  \\ 
    cols $n$: 			&  712 & 1706 & 9133 & 5069 & 8291 & 4486   \\ \hline
    nnz(A) 			& 8755   & 6937  & 19231  & 10839  & 24102   & 14996  \\ \hline
    $\kappa(A)$ 		& 1.1E+2  & 2.1E+4 & 3.7E+2  & 1.2E+4  & 2.8E+3  & 7.8E+3  \\ \hline
    $\|A\|_{2}$ 		& 1.8E+0  & 4.0E+0  & 1.0E+1  & 1.4E+2  & 7.6E+2  & 2.1E+2  \\ \hline
    $\gamma_m(1)$  & 3.0E-3 & 1.1E-1 & 1.1E-1 & 4.2E-3 & 1.6E-3 & 7.1E-3 \\ \hline
    $\gamma_m(5)$  & 3.0E-3 & 2.4E-3 & 8.9E-5 & 5.1E-5 & 3.6E-4 & 1.1E-3 \\ \hline
    $\gamma_m(10)$ & 2.6E-3 & 8.0E-5 & 8.9E-5 & 2.0E-5 & 4.0E-5 & 1.1E-3 \\ \hline
    \end{tabular}
\end{center}
\end{table}

We select six square and six rectangular matrices from other research papers \cite{jia2010refined, liang2014computing} and the University of Florida Sparse Matrix Collections \cite{davis2011university}. Table \ref{ta: Square and rectangular matrices} lists these matrices along with some of their basic properties. 
Among them, the matrices pde2961, dw2048, well1850 and lp\_ganges have relative larger gap ratios and smaller condition number, and thereby are easy ones.
Matrices fidap4, jagmesh8, wang3, deter4, and plddb are hard cases, 
  and matrices lshp3025, ch, and lp\_bnl2 are very hard ones. 
We expect all methods to perform well for solving easy problems. 
Harder problems tend to magnify the difference between methods.

\begin{figure}[h]
  \centering
 \subfigure[Matvec Ratio]{\includegraphics[width=0.49\textwidth]{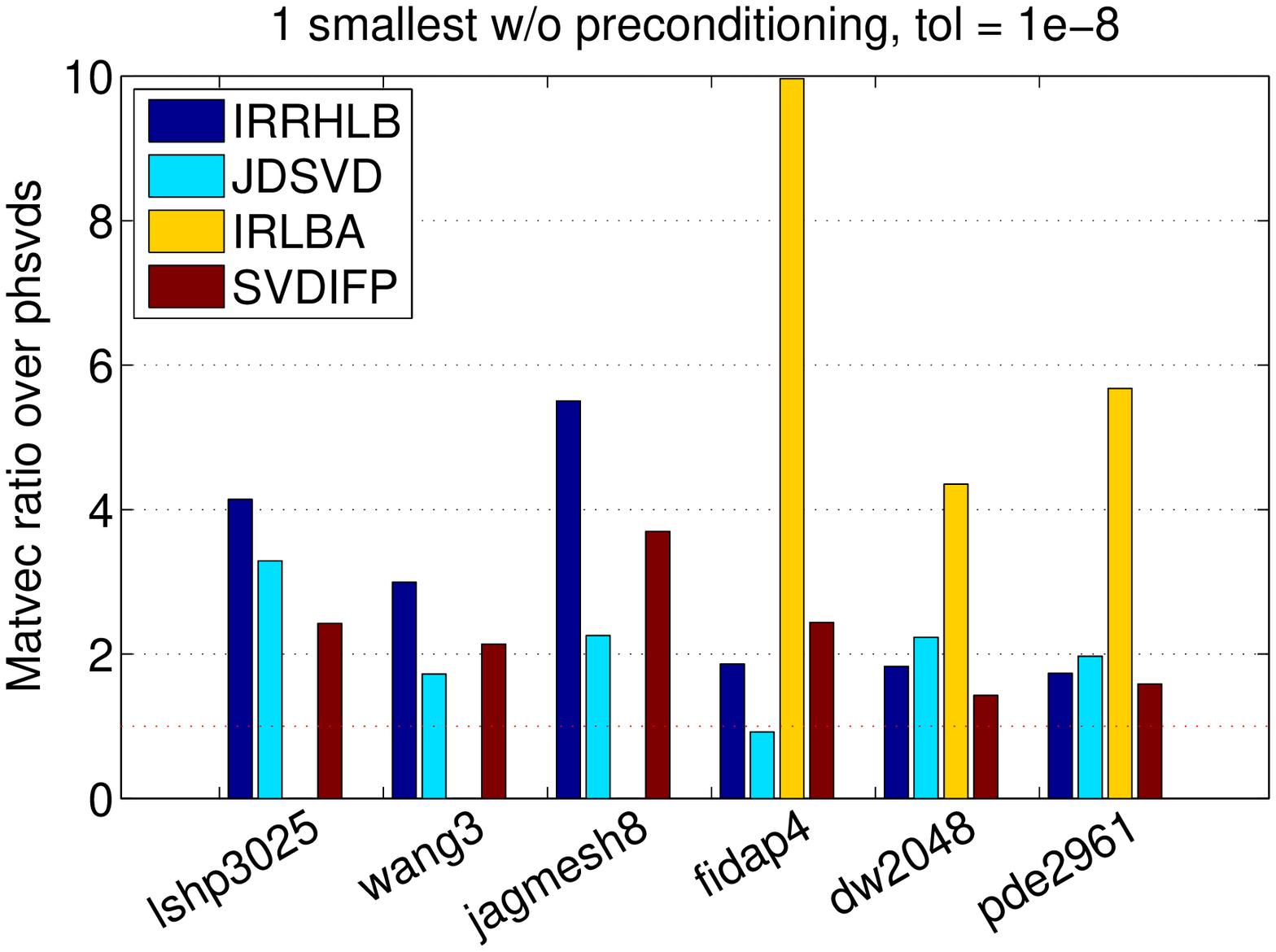}}
 \subfigure[Matvec Ratio]{\includegraphics[width=0.49\textwidth]{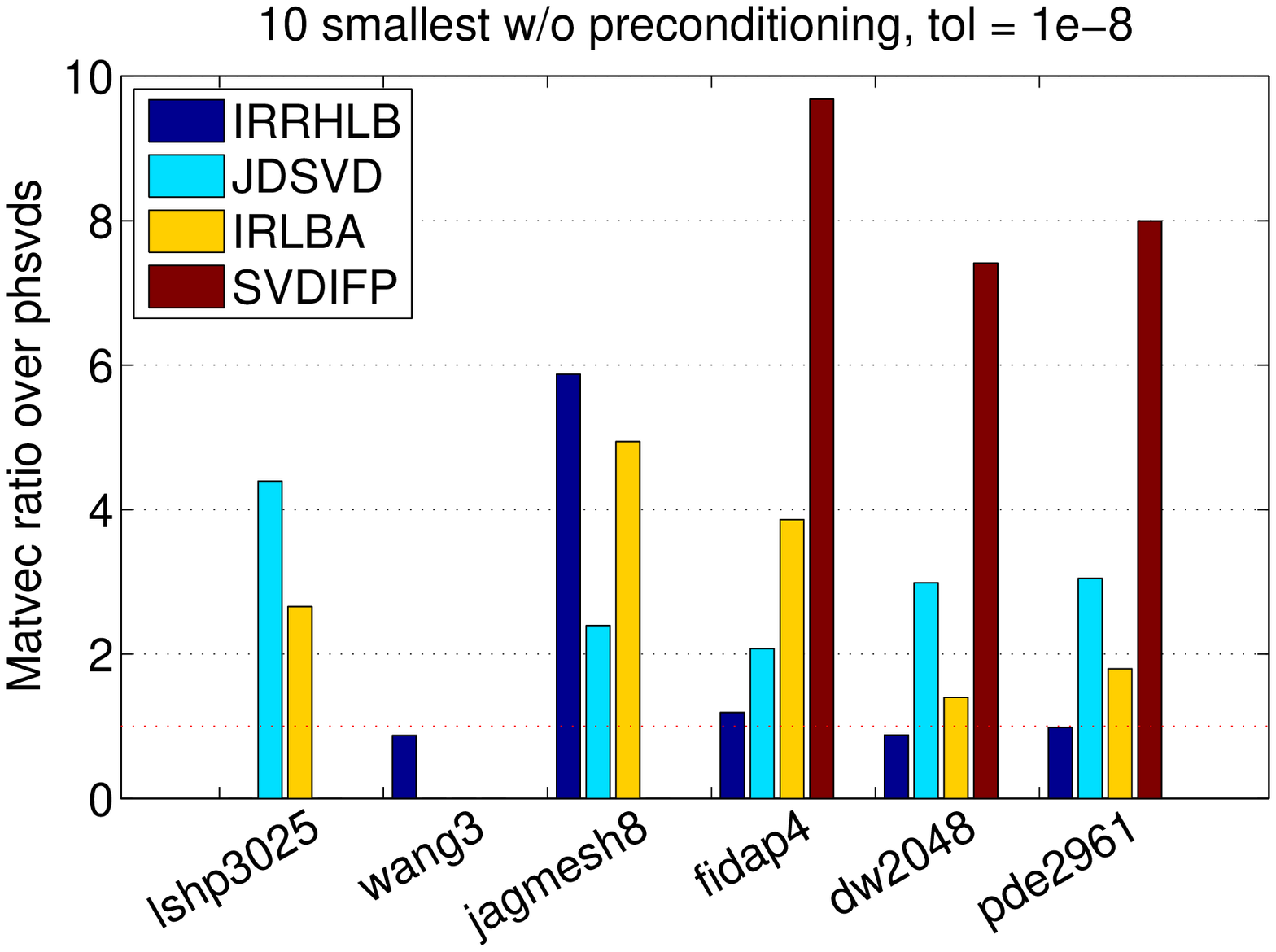}}
 \subfigure[Time Ratio]{\includegraphics[width=0.49\textwidth]{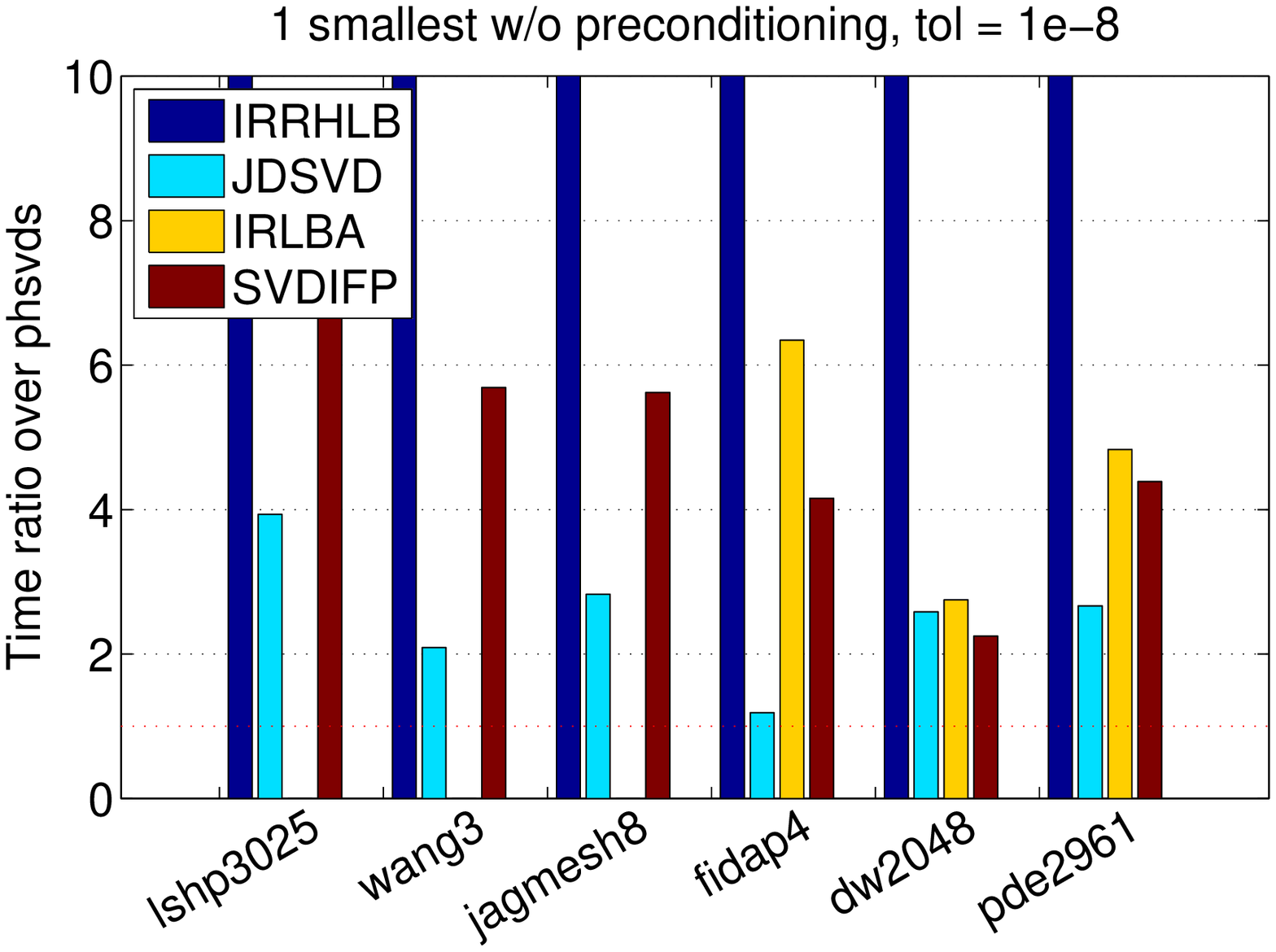}}                
 \subfigure[Time Ratio]{\includegraphics[width=0.49\textwidth]{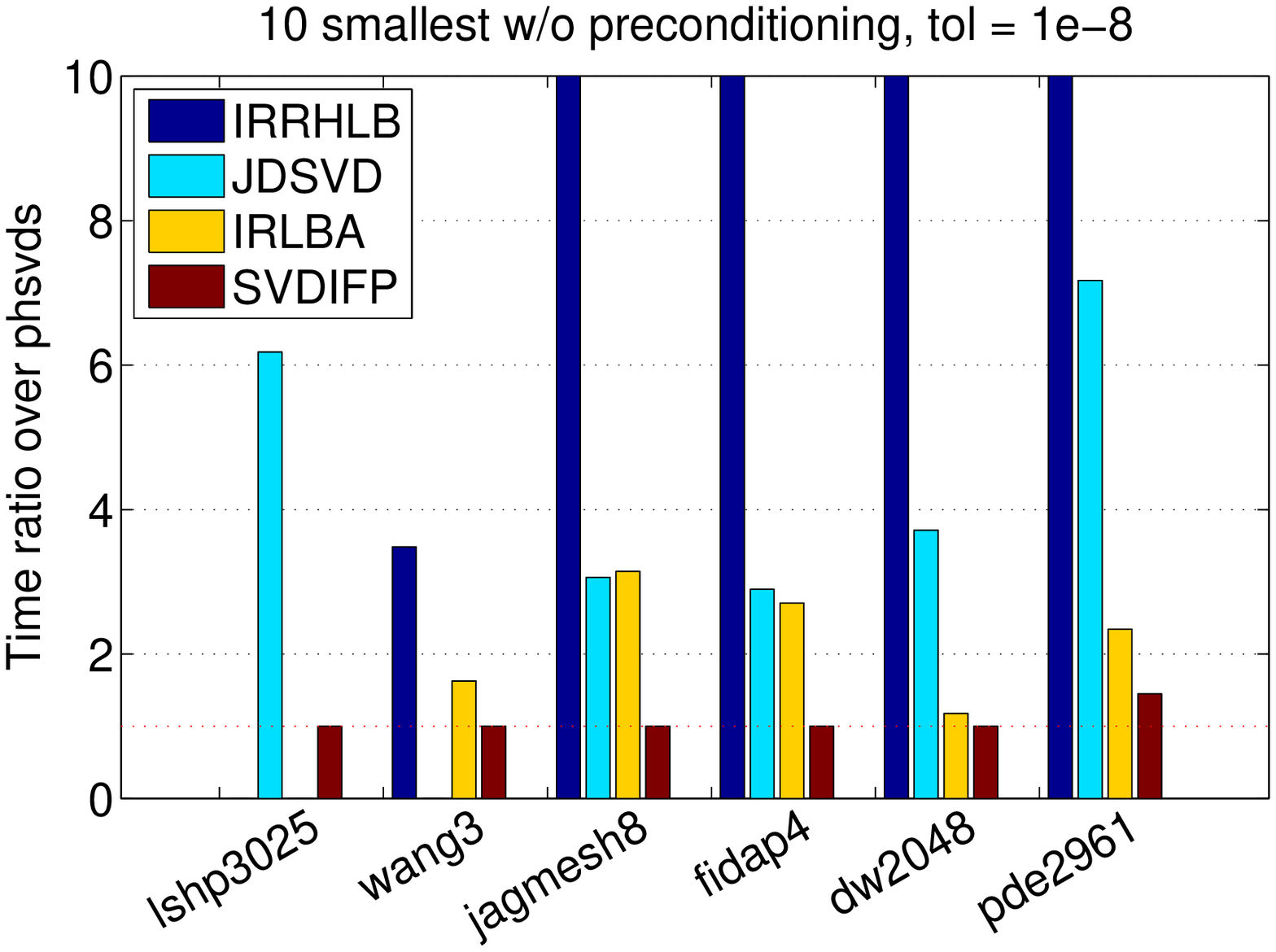}}   
\caption{Matvec and time ratios over PHSVDS(1st stage only) when seeking 1 and 10 smallest singular triplets of square matrices with user tolerance 1E-8.}
\label{fig: ExpB-square-1e-8}
\end{figure}

\begin{figure}[h]
  \centering
 \subfigure[Matvec Ratio]{\includegraphics[width=0.49\textwidth]{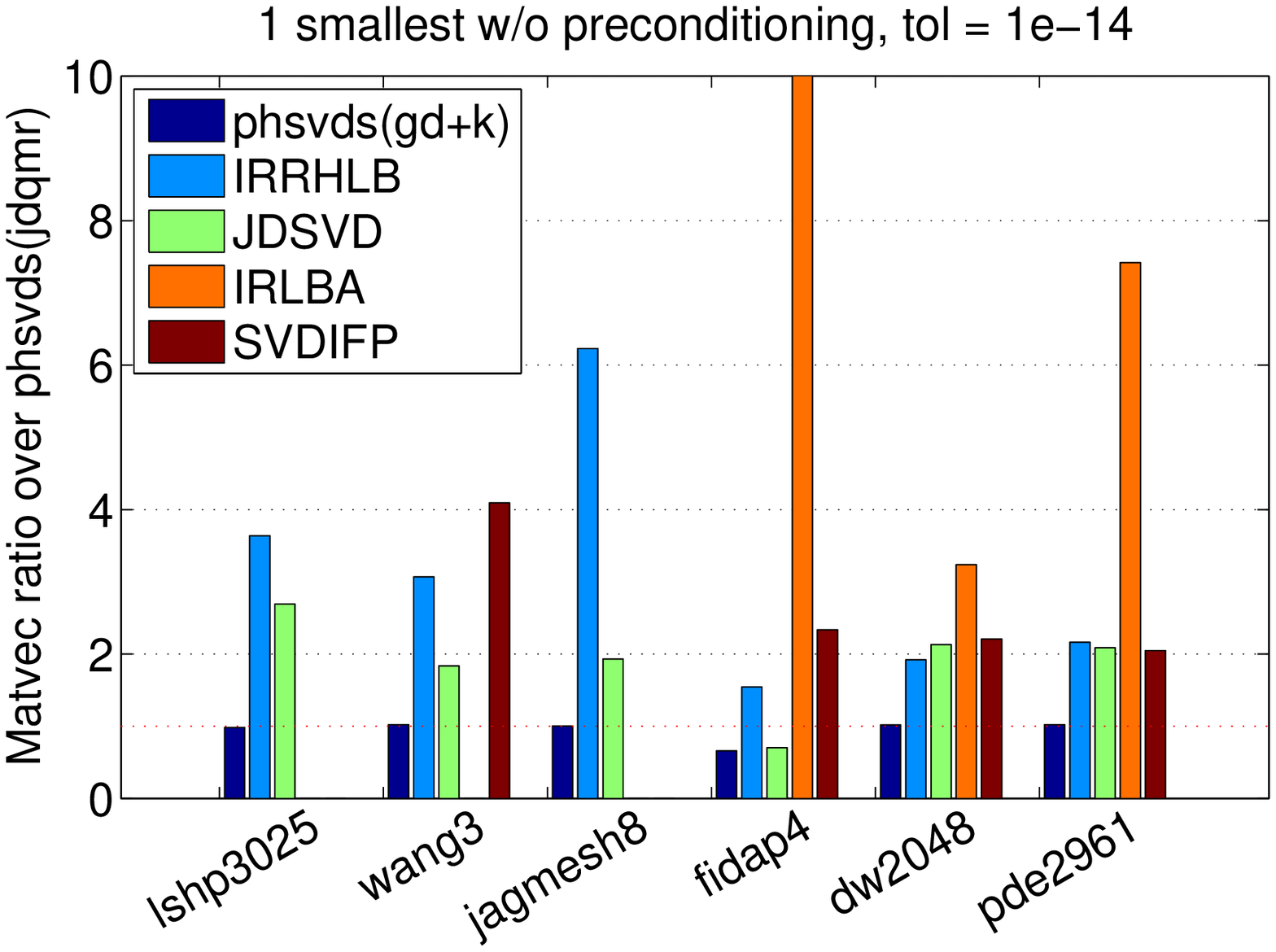}}
 \subfigure[Matvec Ratio]{\includegraphics[width=0.49\textwidth]{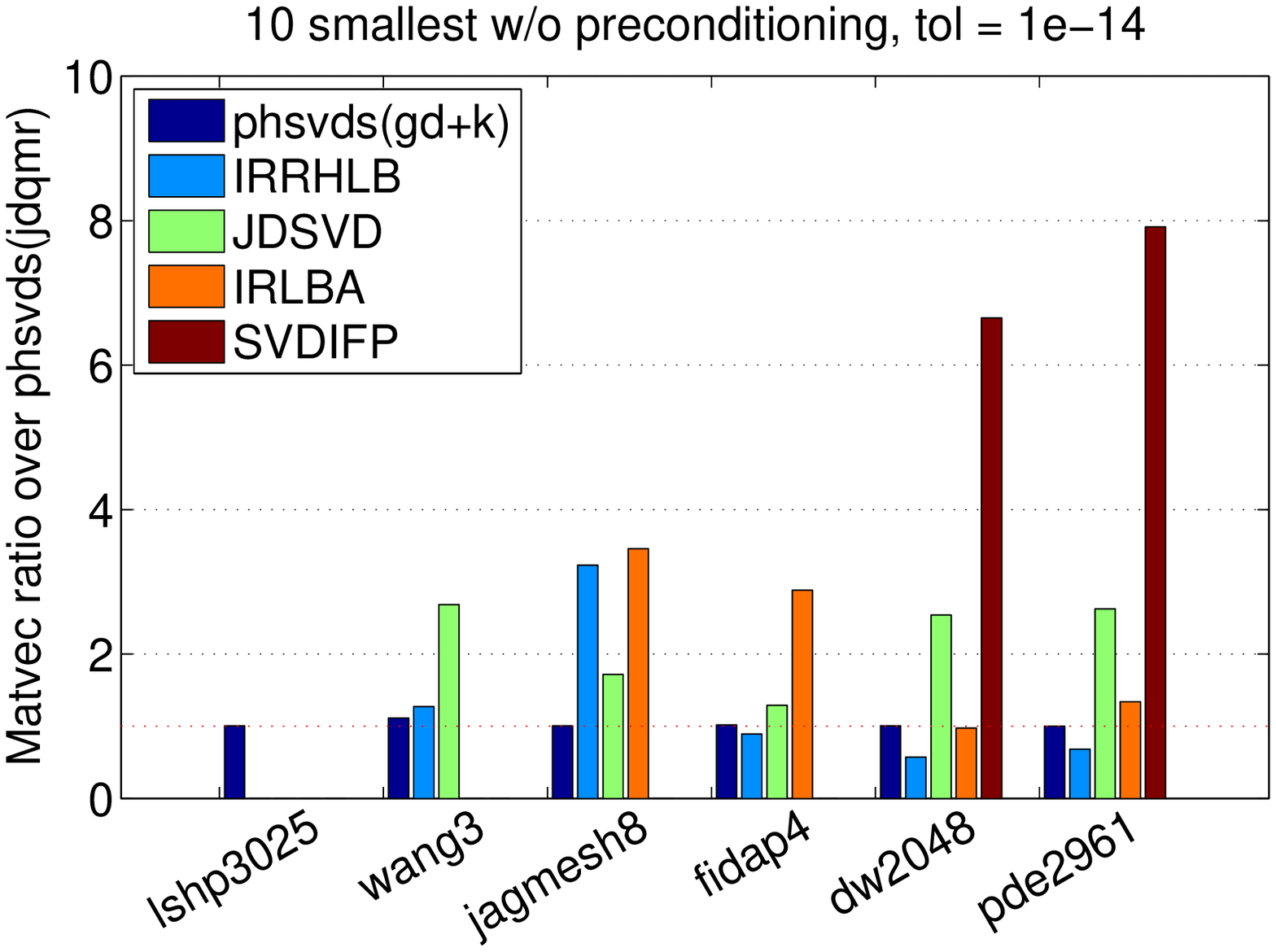}}
 \subfigure[Time Ratio]{\includegraphics[width=0.49\textwidth]{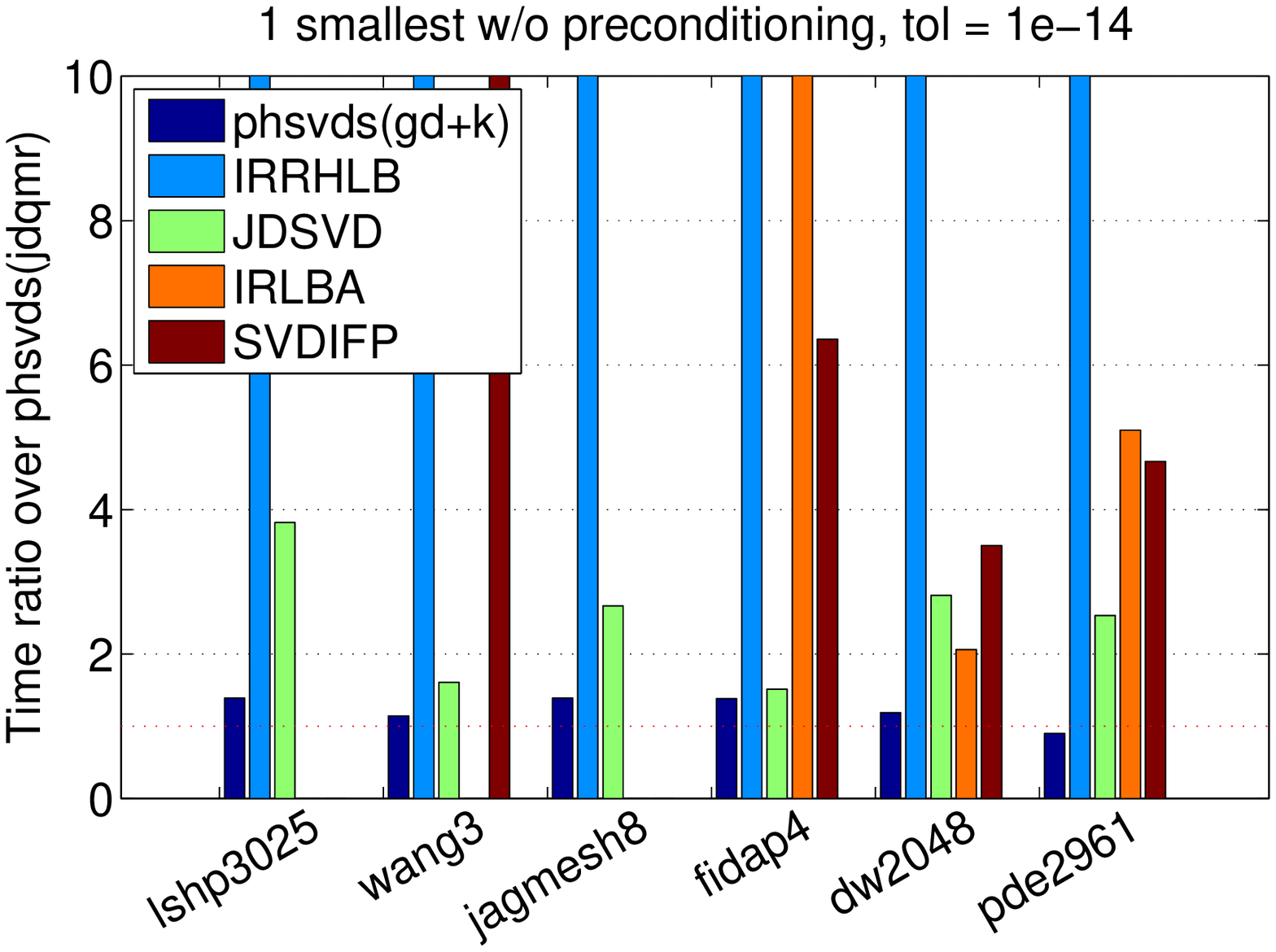}}                
 \subfigure[Time Ratio]{\includegraphics[width=0.49\textwidth]{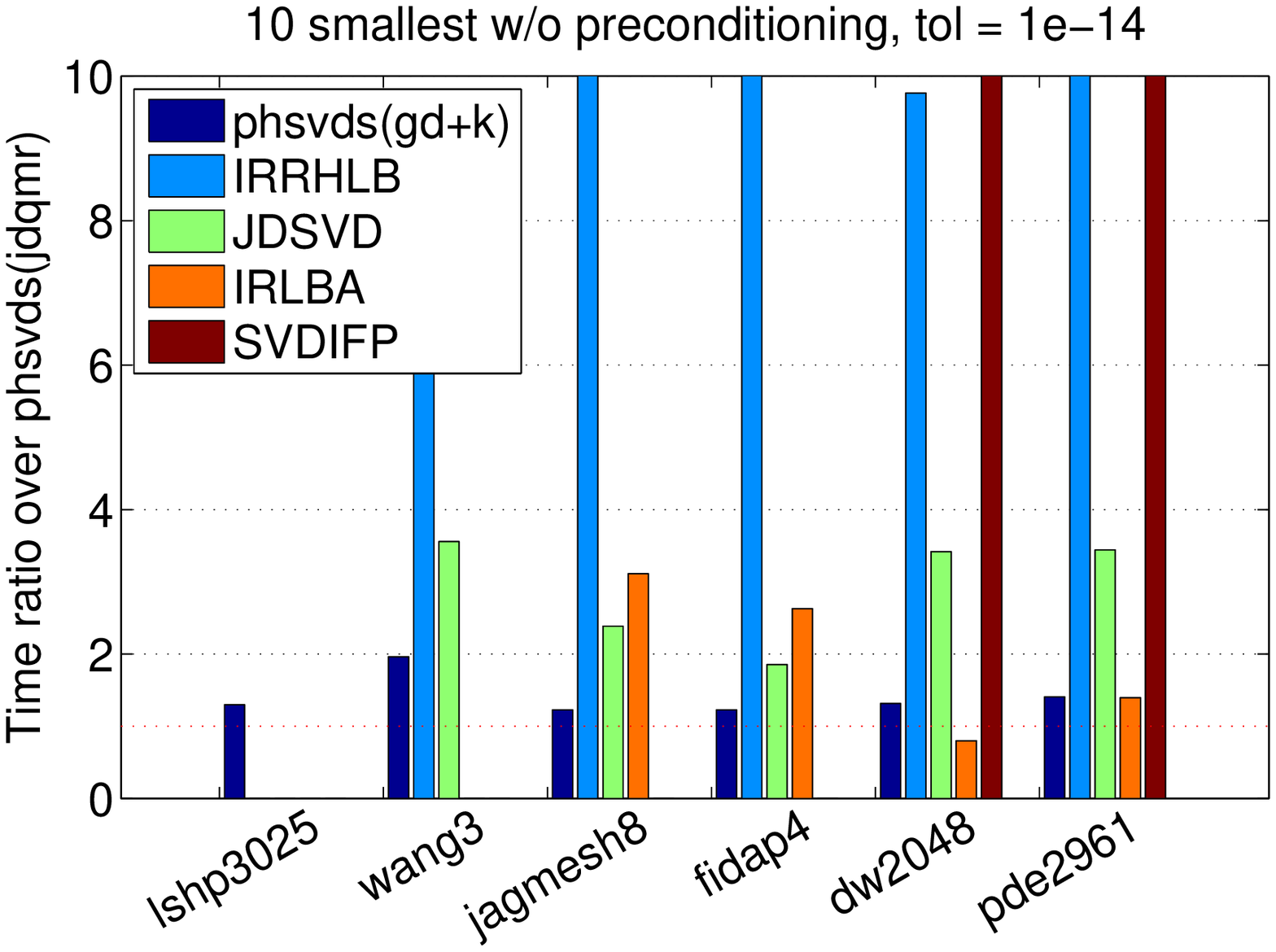}}   
\caption{Matvec and time ratios over PHSVDS(JDQMR) when seeking 1 and 10 smallest singular triplets of square matrices with user tolerance 1E-14. 
The PHSVDS(GD+k) variant uses the GD+k eigenmethod in the second stage.}
\label{fig: ExpB-square-1e-14}
\end{figure}

\begin{figure}[h]
  \centering
 \subfigure[Matvec Ratio]{\includegraphics[width=0.49\textwidth]{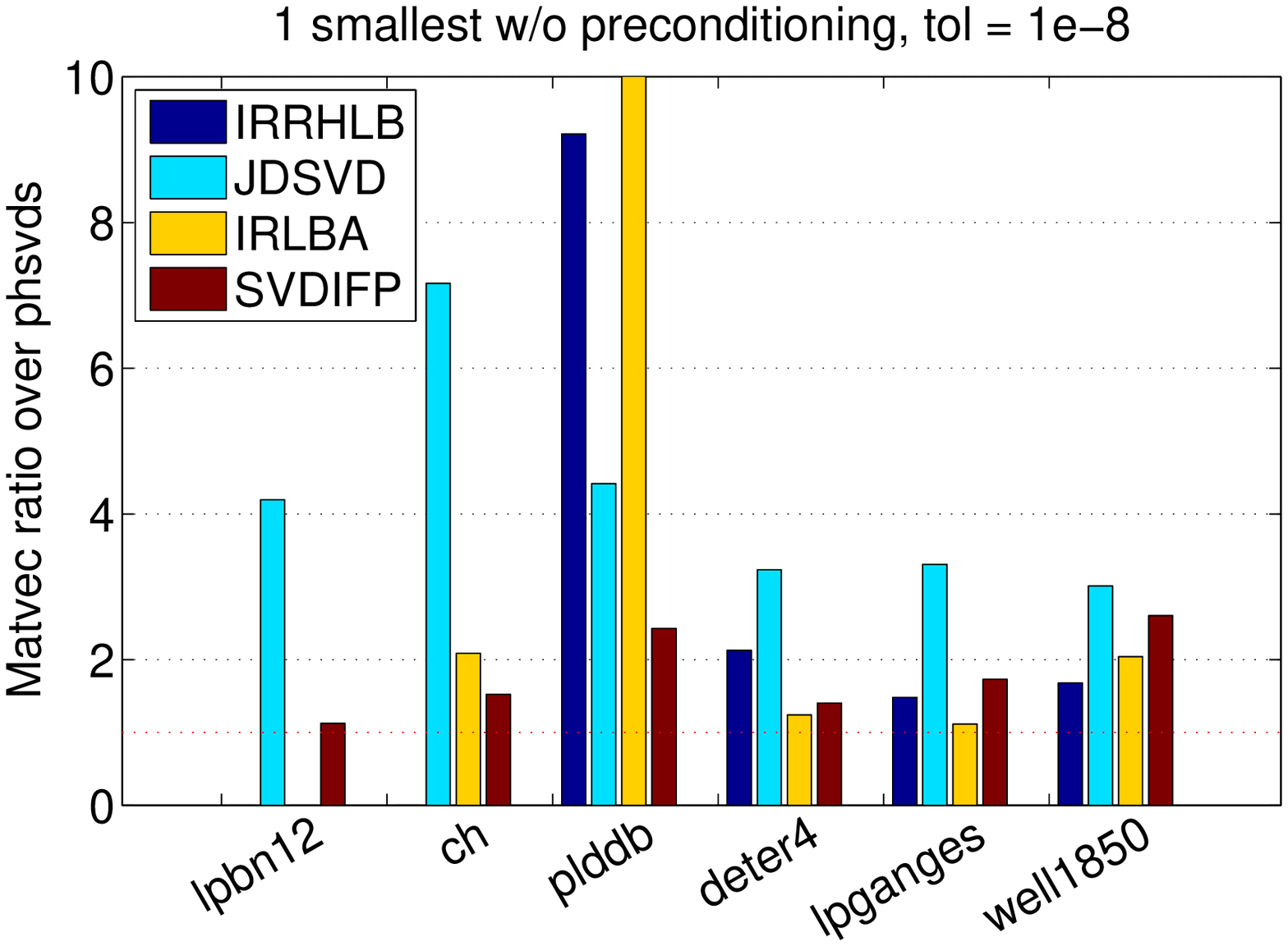}}
 \subfigure[Matvec Ratio]{\includegraphics[width=0.49\textwidth]{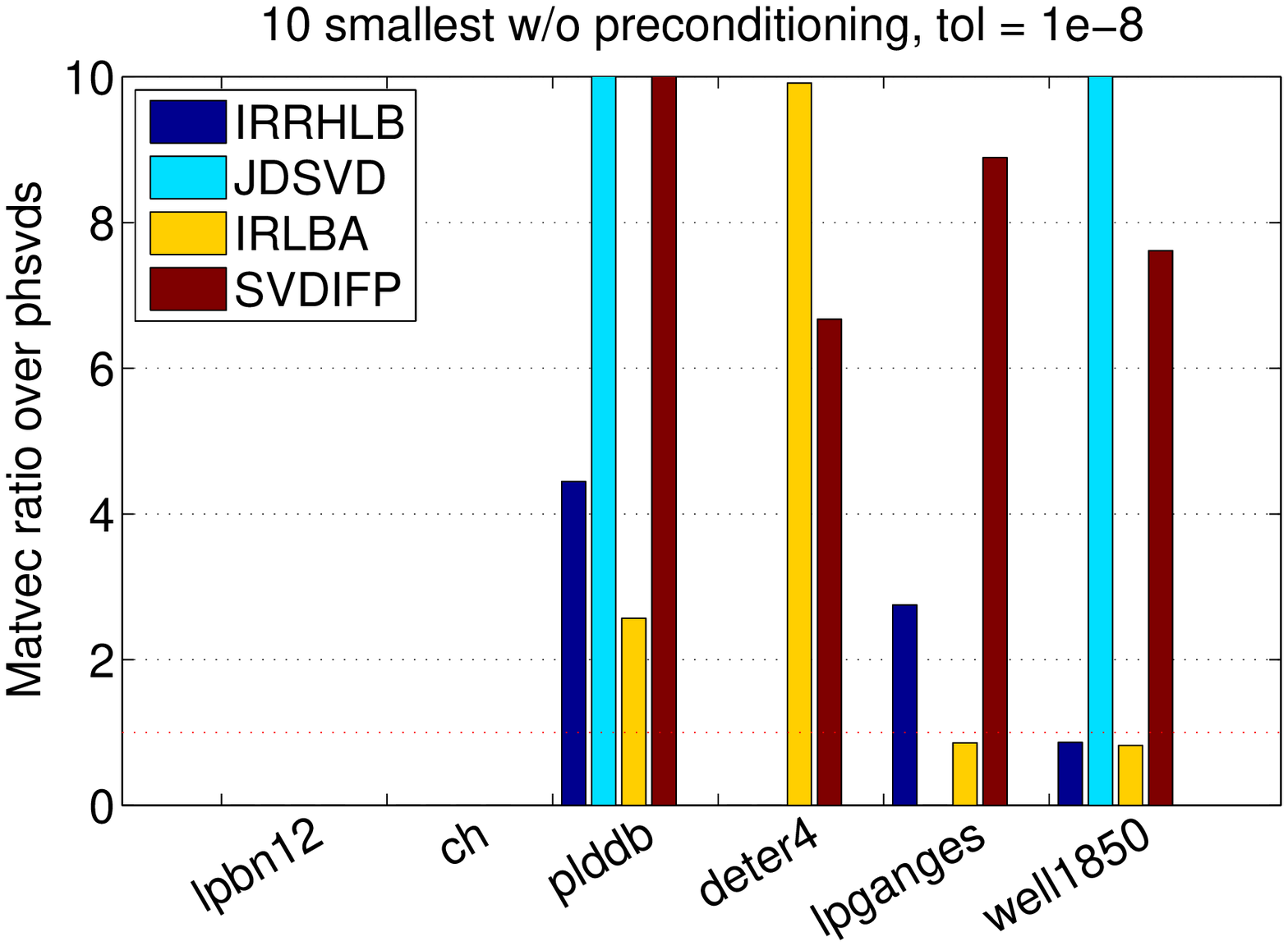}}
 \subfigure[Time Ratio]{\includegraphics[width=0.49\textwidth]{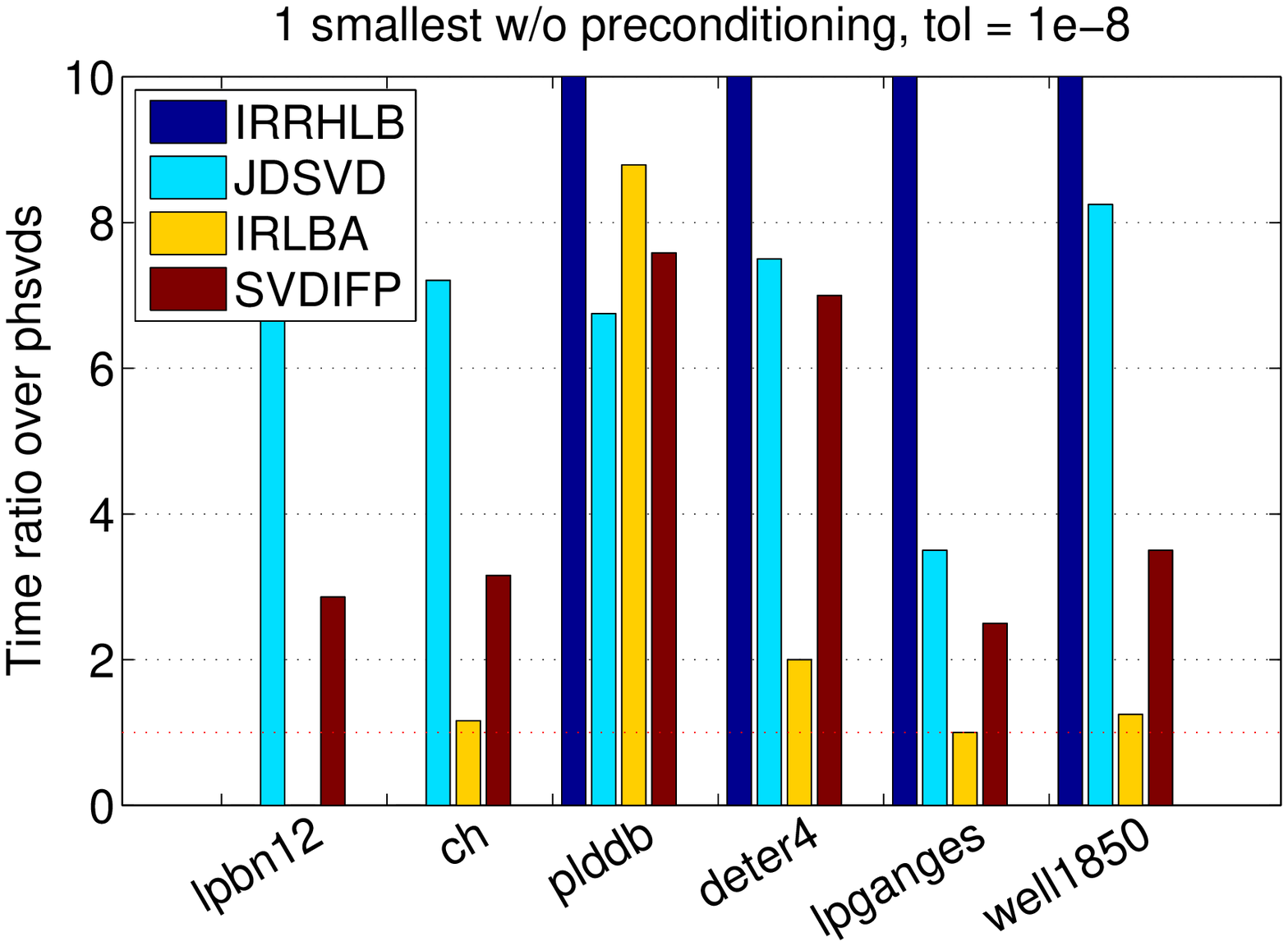}}                
 \subfigure[Time Ratio]{\includegraphics[width=0.49\textwidth]{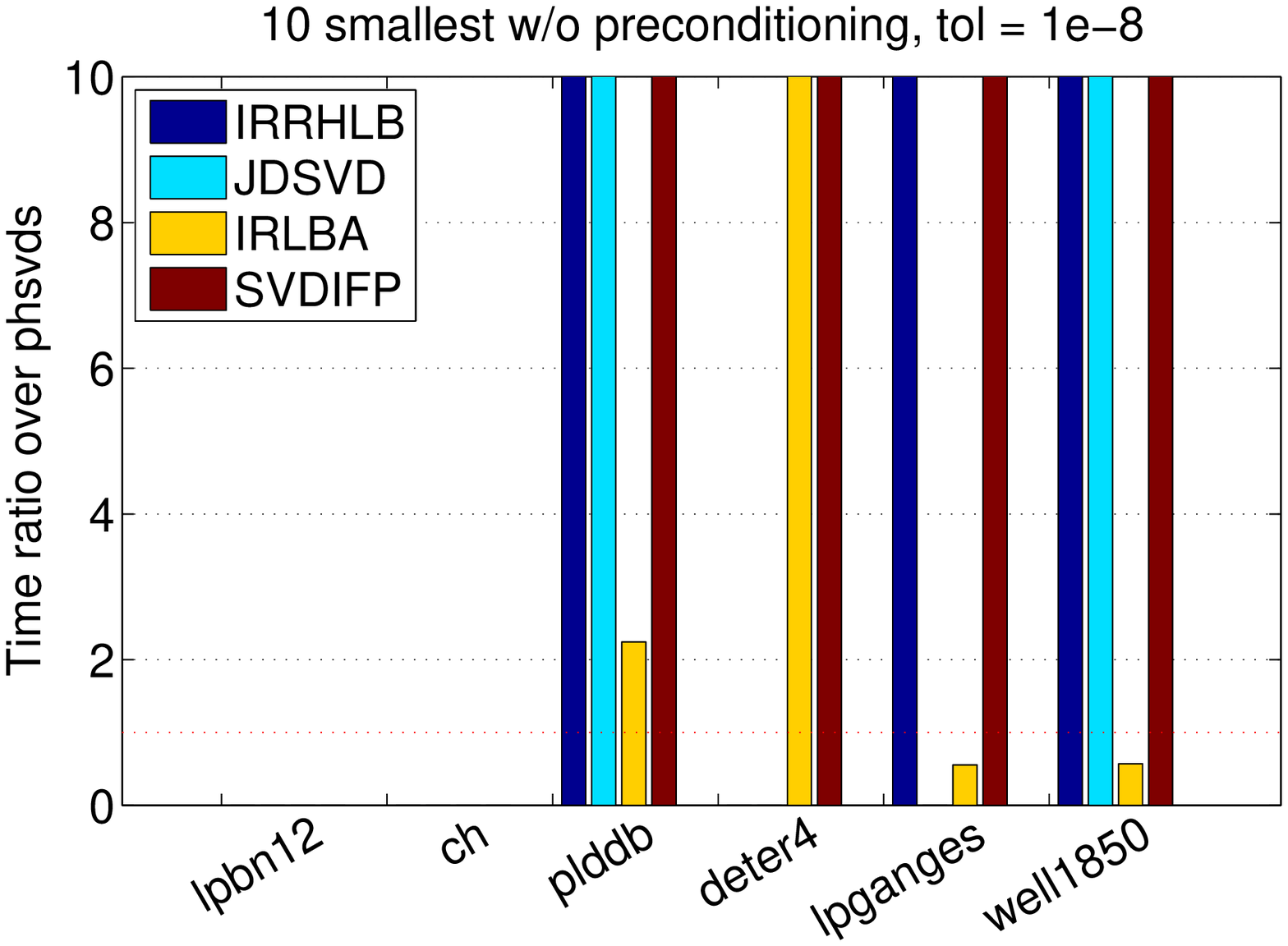}}   
\caption{Matvec and time ratios over PHSVDS(1st stage only) when seeking 1 and 10 smallest singular triplets of rectangular matrices with user tolerance 1E-8.}
\label{fig: ExpB-rectangular-1e-8}
\end{figure}
  
\begin{figure}[h]
  \centering
 \subfigure[Matvec Ratio]{\includegraphics[width=0.49\textwidth]{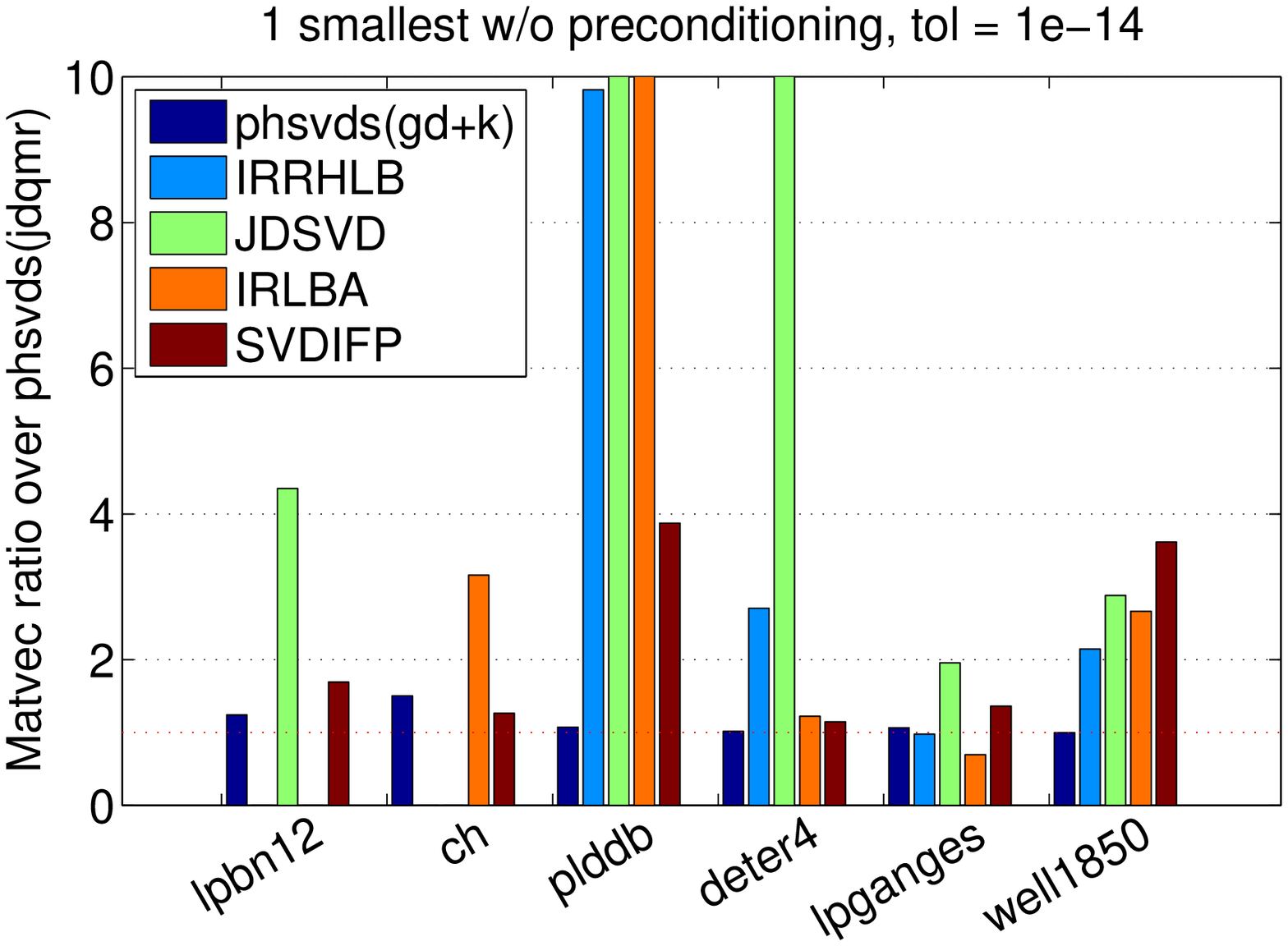}}
 \subfigure[Matvec Ratio]{\includegraphics[width=0.49\textwidth]{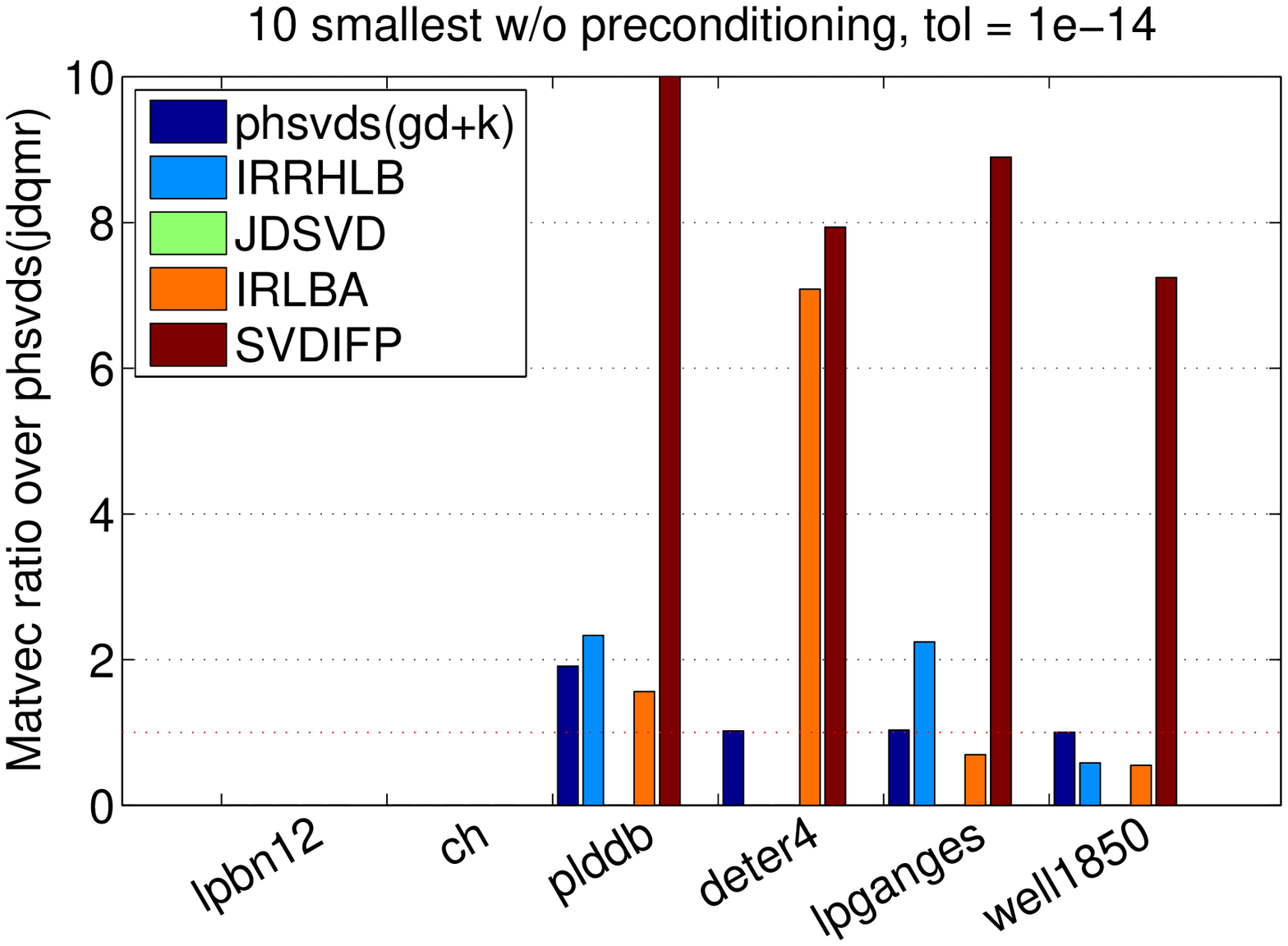}}
 \subfigure[Time Ratio]{\includegraphics[width=0.49\textwidth]{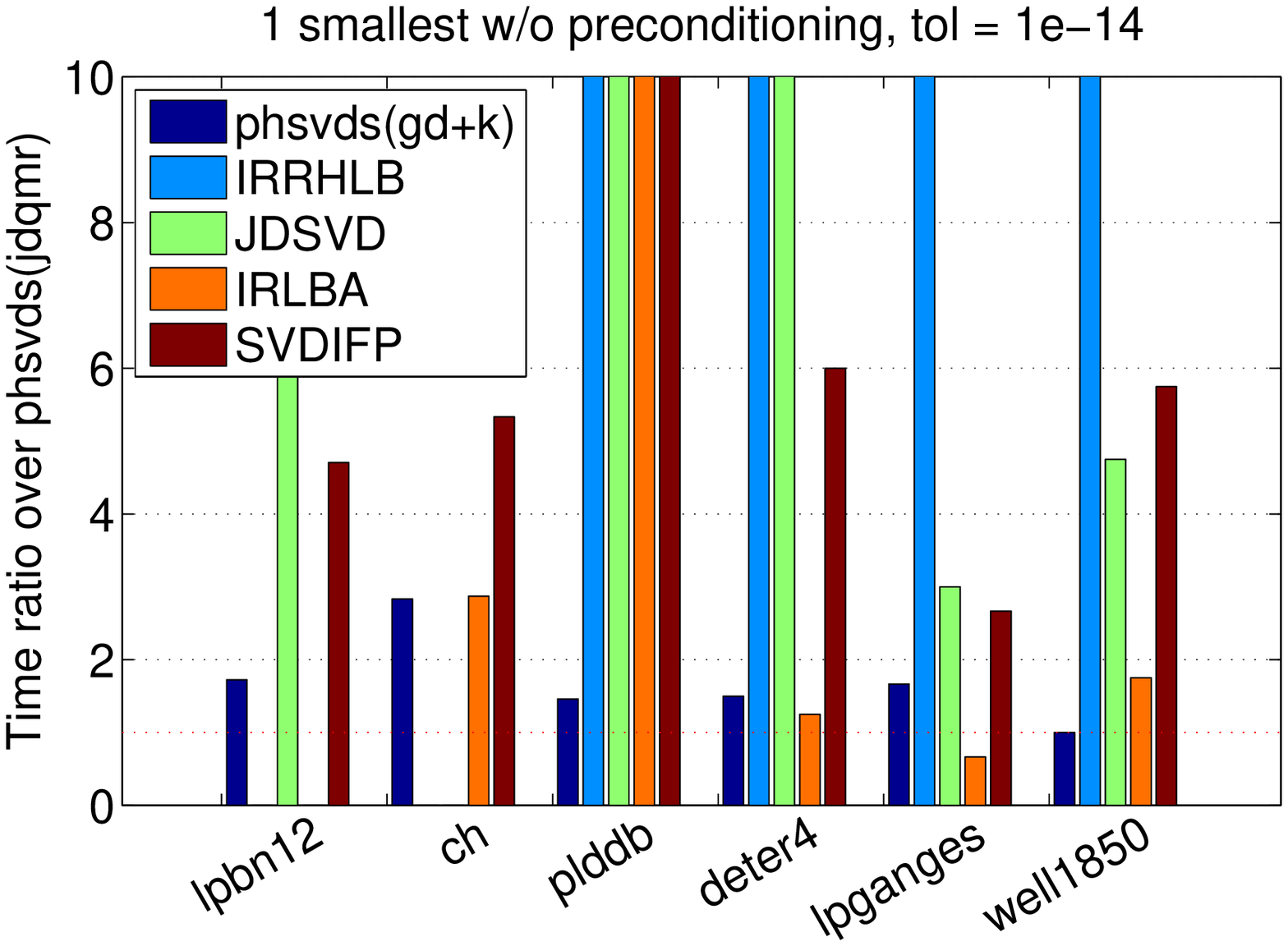}}                
 \subfigure[Time Ratio]{\includegraphics[width=0.49\textwidth]{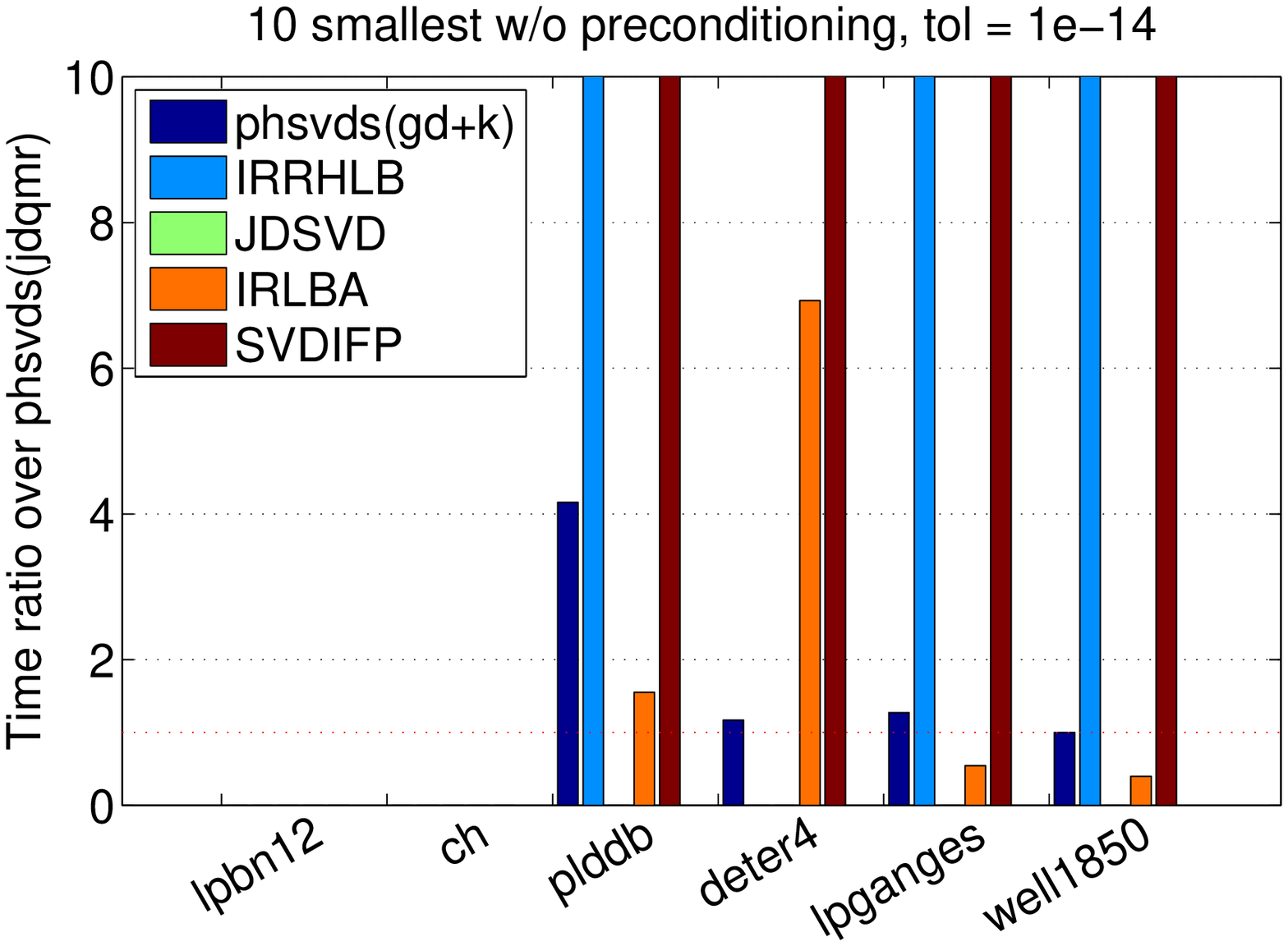}}   
\caption{Matvec and time ratios over PHSVDS(JDQMR) when seeking 1 and 10 smallest singular triplets of rectangular matrices with user tolerance 1E-14. 
The PHSVDS(GD+k) variant uses the GD+k eigenmethod in the second stage.}
\label{fig: ExpB-rectangular-1e-14}
\end{figure}

Figures \ref{fig: ExpB-square-1e-8}, \ref{fig: ExpB-square-1e-14}, \ref{fig: ExpB-rectangular-1e-8} and \ref{fig: ExpB-rectangular-1e-14} show that PHSVDS variants converge faster and more robustly than all other methods on both square and rectangular matrices. Specifically, Figure \ref{fig: ExpB-square-1e-8} shows that for moderate accuracy the normal equations solved with GD+k are significantly faster. 
For instance, PHSVDS is at least two or three times faster than other methods when solving hard problems for any number of smallest singular values. 
In fact, we have noticed that even for moderate accuracy, 
 all other methods are challenged by hard problems, where they are often inefficient or even fail to converge to all desired singular values. 
When solving easy problems, still PHSVDS is faster than other methods
  and only IRRHLB can be competitive when seeking 10 singular values.
This better global convergence for many eigenvalues is typical of the Lanczos method. The superiority of PHSVDS is a result of using a near-optimal eigenmethod.

Figure \ref{fig: ExpB-square-1e-14} shows that the PHSVDS(GD+k) variant
  is comparable in terms of matvecs to PHSVDS(JDQMR), which is the base 
  of the ratios, but the JDQMR typically requires less time if the matrix 
  is sparse enough.
It also shows that despite the slower convergence on the augmented 
  matrix in stage two, the higher accuracy requirement does not help 
  the rest of the methods.
For computing 10 eigenpairs, IRRHLB shows a small edge in the number of 
  iterations for two easy cases.
However, PHSVDS method never misses eigenvalues, is consistently much faster 
 than all other methods, and significantly faster than IRRHLB in hard cases.
SVDIFP is also not competitive, partly due to its inefficient restarting strategy. 
Interestingly, not only does PHSVDS enjoy better robustness but also its 
  execution time is ten times faster than IRRHLB for the cases where 
  IRRHLB takes fewer matvecs.

Figures \ref{fig: ExpB-rectangular-1e-8} and \ref{fig: ExpB-rectangular-1e-14} show that the advantage of PHSVDS is even more significant on rectangular matrices. For example, except for the two easy problems well1850 and lp\_ganges, PHSVDS is often five or ten times faster than the other methods. 
The reason is the beneficial use of the first stage, 
  but also because PHSVDS works on $C$ with dimension $\min(m,n)$, 
  which saves memory and computational costs. 
SVDIFP also shares this advantage.
Interestingly, JDQMR converges much faster than GD+k on some hard problems 
  such as plddb, ch and lp\_bnl2 in Figure \ref{fig: ExpB-rectangular-1e-14}.
The reason is the availability of excellent shifts from the first stage.
We conclude that PHSVDS is the fastest method and sometimes the 
  only method that converges for hard problems without preconditioning.

\subsection{With preconditioning}
The previous figures show the remarkable difficulty of
  solving for the smallest singular values.
Preconditioning is a prerequisite for practical problems,
   which limits our choice to PHSVDS, JDSVD and SVDIFP. 

We first compare our two stage method and our dynamic two-stage method for 
  two different quality preconditioners.  
We choose $M=LU$, the factorization obtained from MATLAB's ILU function 
  on a square matrix $A$ with parameters {\tt 'type=ilutp'}, 
  {\tt 'thresh=1.0'}, and 
  varying {\tt 'droptol=1E-2'} or {\tt 'droptol=1E-3'}. 
Given these two $M$, we form the preconditioners for PHSVDS as
   $M^{-1}M^{-T}$ and $[0 \ M^{-1};M^{-T} \ 0]$.
Without loss of generality, PHSVDS chooses GD+k for both stages.
We seek ten smallest singular values with tolerance 1E-14. 

As shown in Table \ref{ta: ExpC-dynamic}, both variants of PHSVDS 
  can solve the problems effectively with a good preconditioner ({\tt'droptol=1E-3'}).
In this case, the static two stage method is always better than the dynamic one
  because of the overhead incurred by switching between the two methods.
On the other hand, when using the preconditioner with {\tt'droptol=1E-2'}, 
  the two stage PHSVDS is slower than the dynamic in some cases, 
  and in the case of lshp3025, much slower. 
The reason is the inefficiency of the preconditioner in the normal equations.
Our dynamic PHSVDS can detect the convergence rate 
  difference and choose the faster method to accomplish the remaining computations.
Of course, if this issue is known beforehand, users can bypass the dynamic 
  heuristic and call directly the second stage.

\begin{table}[htbp]
\centering
\caption{Dynamic stage switching PHSVDS (D) vs two stage PHSVDS (P) with
two preconditioners:  (H)igh quality ILU(1e-3) and (L)ow quality ILU(1E-2).
We seek 10 smallest triplets with $\delta = $1E-14.}

\label{ta: ExpC-dynamic}
\small
\begin{center}
	\begin{tabular}{|c|c|rr|rr|rr|rr|rr|rr|}
	\hline
	  \multicolumn{2}{|r}{}
	 & \multicolumn{2}{c}{{\tt pde2961}}
	 & \multicolumn{2}{c}{{\tt dw2048}} 
	 & \multicolumn{2}{c}{{\tt fidap4}}
	 & \multicolumn{2}{c}{{\tt jagmesh8}}
	 & \multicolumn{2}{c}{{\tt wang3}} 
	 & \multicolumn{2}{c|}{{\tt lshp3025}}       \\  \hline
	     &     & MV & Sec &  MV  &  Sec &  MV &  Sec 
	               & MV & Sec &  MV  &  Sec &  MV &  Sec\\ \hline
	H    & P
	     & 166  & 0.4 & 211 & 0.7 & 210 & 1.3 
	     & 163  & 0.5 & 306 & 5.5 & 209 & 3.0\\ 
	H    & D
	     & 242  & 0.4 & 283 & 0.7 & 286 & 1.4
	     & 223  & 0.6 & 396 & 5.5 & 273 & 3.6 \\ \hline
	L    & P
	     & 258 & 0.5 & 673 & 1.6 & 813 & 3.2
	     & 990 & 3.1 & 736 & 8.9 & 7631 & 132 \\ 
	L    & D
	     & 307 & 0.5 & 668 & 1.5 & 1043 & 3.7 
	     & 547 & 1.6 & 1038 & 9.6 & 696 & 10 \\ \hline
	\end{tabular}
\end{center}
\end{table}

Next, we compare the two stage PHSVDS with JDSVD and SVDIFP with a
  good quality preconditioner.
Except for preconditioning, all other parameters remain as before.
For the first preconditioner we use MATLAB's ILU on a square matrix $A$.
For the second preconditioner we use the RIF MEX function provided 
  in \cite{liang2014computing} on a rectangular matrix with {\tt'droptol=1E-3'}.
The resulting RIF factors $LDL^{T}\approx A^TA$, where $D$ is diagonal matrix
  with 0 and 1 elements, are used to construct the pseudoinverses
  $M^{-1} = L^{-T} L^{-1} A^T $ and $M^{-T} = A L^{-T} L^{-1}$ for 
  preconditioning the second stage of PHSVDS and JDSVD. 
To obtain uniform behavior across methods, we disable in SVDIFP 
  the parameter 'COLAMD', which computes an approximate minimum 
  degree column permutation to obtain sparser LU factors.
For JDSVD, we try both enabling and disabling the initial Krylov 
  subspace and report the best result.

\begin{table}[htbp]
\centering
\caption{Seeking 1 and 10 smallest singular triplets with ILU, droptol = 1E-3. 
We report results from both PHSVDS variants: PHSVDS(GD+k) and PHSVDS(JDQMR). 
We report separately the time for generating the preconditioner and the time 
  for running each method.}
\label{ta: ExpC-ILU-1e-3}
\small
\begin{center}
	\begin{tabular}{|l|l|rr|rr|rr|rr|}
	\hline
	 & \multicolumn{1}{|r}{$\delta =$ 1E-8 \hfill Matrix:}
	 & \multicolumn{2}{c}{{\tt fidap4}}   
	 & \multicolumn{2}{c}{{\tt jagmesh8}}
	 & \multicolumn{2}{c}{{\tt wang3}} 
	 & \multicolumn{2}{c|}{{\tt lshp3025}}       \\  \hline
	 & \multicolumn{1}{|r}{\hfill ILU Time:} 
	 & \multicolumn{2}{c}{{\tt 0.1}}   
	 & \multicolumn{2}{c}{{\tt 0.1}}
	 & \multicolumn{2}{c}{{\tt 2.8}} 
	 & \multicolumn{2}{c|}{{\tt 0.1}}       \\  \hline
	$k$ & Method & MV & Sec & MV & Sec &  MV  &  Sec &  MV   &  Sec \\ \hline
	1 &  {\footnotesize PHSVDS(1st stage only)} 
	     & 15 & 0.1 & 13 & 0.1 & 46 & 2.5 & 19 & 0.3 \\ 
	1 &  {\footnotesize JDSVD}
	     & 67 & 0.7 & 34 & 0.3 & 45 & 2.0 & 56 & 1.5 \\ 
	1 &  {\footnotesize SVDIFP}
	     & 58 & 0.4 & 51 & 0.2 & 132 & 5.7 & 82 & 1.7 \\ \hline	      
    10 &  {\footnotesize PHSVDS(1st stage only)} 
	     & 117 & 0.6 & 91 & 0.2 & 185 & 2.5 & 122 & 1.7 \\ 
	10 & {\footnotesize JDSVD}
	     & 342 & 3.1 & 287 & 1.4 & 320 & 15.7 & 364 & 10.5 \\ 	
	10 & {\footnotesize SVDIFP}
	     & 691 & 3.0 & 561 & 1.2 & 1179 & 29.1 & 1187 & 21.9 \\ \hline		     
	\end{tabular}
\end{center}	     
\small
\begin{center}
	\begin{tabular}{|l|l|rr|rr|rr|rr|}
	\hline
	 & \multicolumn{1}{|r}{$\delta =$ 1E-14 \hfill Matrix:}
	 & \multicolumn{2}{c}{{\tt fidap4}}   
	 & \multicolumn{2}{c}{{\tt jagmesh8}}
	 & \multicolumn{2}{c}{{\tt wang3}} 
	 & \multicolumn{2}{c|}{{\tt lshp3025}}       \\  \hline
	 & \multicolumn{1}{|r}{\hfill ILU Time:} 
	 & \multicolumn{2}{c}{{\tt 0.1}}   
	 & \multicolumn{2}{c}{{\tt 0.1}}
	 & \multicolumn{2}{c}{{\tt 2.8}} 
	 & \multicolumn{2}{c|}{{\tt 0.1}}       \\  \hline 
	1 &  {\footnotesize PHSVDS(GD+k)} 
	     & 62 & 0.6 & 52 & 0.1 & 102 & 1.8 & 66 & 0.5 \\ 
	1 &  {\footnotesize PHSVDS(JDQMR)} 
	     & 64 & 0.3 & 55 & 0.1 & 106 & 1.1 & 68 & 0.5 \\
	1 &  {\footnotesize JDSVD}
	     & 78 & 1.5 & 45 & 0.3 & 67 & 3.0 & 79 & 1.2 \\ 
	1 &  {\footnotesize SVDIFP}
	     & 98 & 0.6 & 100 & 0.3 & 235 & 8.3 & 159 & 2.3 \\ \hline	 	          
	10 &  {\footnotesize PHSVDS(GD+k)} 
	     & 210 & 1.3 & 163  & 0.5 & 306 & 5.5 & 209 & 3.0 \\ 
	10 &  {\footnotesize PHSVDS(JDQMR)} 
	     & 251 & 1.2 & 215 & 0.5 & 402 & 5.0 & 265 & 3.5 \\
	10 &  {\footnotesize JDSVD}
	     & 573 & 5.5 & 408 & 1.9 & 518 & 14.6 & 606 & 26.9 \\ 	
	10 &  {\footnotesize SVDIFP}
	     & 1152 & 5.4 & 981 & 1.9 & 1991 & 51.4 & 1897 & 29.1 \\ \hline		       	     	     
	\end{tabular}
\end{center}
\end{table}

\begin{table}[htbp]
\centering
\caption{Seeking 1 and 10 smallest singular triplets with RIF, droptol = 1E-3. 
We report results from both PHSVDS variants: PHSVDS(GD+k) and PHSVDS(JDQMR). 
We report separately the time for generating the preconditioner and the time 
  for running each method.}
\label{ta: ExpC-RIF-1e-3}
\small
\begin{center}
	\begin{tabular}{|l|l|rr|rr|rr|rr|rr|}
	\hline
	 \multicolumn{2}{|r}{$\delta =$ 1E-8 \hfill Matrix:}
	 & \multicolumn{2}{c}{{\tt fidap4}}   
	 & \multicolumn{2}{c}{{\tt jagmesh8}}      
	 & \multicolumn{2}{c}{{\tt deter4}}   
	 & \multicolumn{2}{c}{{\tt plddb}}
	 & \multicolumn{2}{c|}{{\tt lp\_bnl2}}       \\  \hline
	 & \multicolumn{1}{|r}{\hfill RIF Time:} 
	 & \multicolumn{2}{c}{{\tt 1.5}}   
	 & \multicolumn{2}{c}{{\tt 0.5}}       
	 & \multicolumn{2}{c}{{\tt 11.0}}   
	 & \multicolumn{2}{c}{{\tt 0.4}}
	 & \multicolumn{2}{c|}{{\tt 1.6}}       \\  \hline		 	 
	$k$ & Method & MV & Sec &  MV  &  Sec & MV & Sec & MV & Sec &  MV  &  Sec \\ \hline
	1 &  {\footnotesize PHSVDS} 
	     & 291 & 0.4 & 119 & 0.1 & 27 & 0.2 & 10 & 0.1 & 15 & 0.1 \\ 
	1 &  {\footnotesize JDSVD}
	     & 1729 & 8.9 & 1311 & 3.6  & 122 & 3.8 & 67 & 0.3 & 89 & 0.5 \\ 
	1 &  {\footnotesize SVDIFP}
	     & 513 & 1.4 & 622 & 0.9 & 142 & 2.1 & 29 & 0.1 & 49 & 0.1 \\ \hline	
	10 &  {\footnotesize PHSVDS} 
	     & 1224 & 1.8 & 307 & 0.5 & 405 & 2.4 & 52 & 0.1 & 74 & 0.1 \\ 
	10 &  {\footnotesize JDSVD}
	     & 8131 & 39.5 & 2356 & 5.7 & -- & -- & -- & -- & -- & -- \\ 
	10 &  {\footnotesize SVDIFP}
	     & 4359 & 10.3 & 3118 & 4.2  & 2278 & 45.4 & 390 & 1.0 & 453 & 1.0 \\ \hline		           
	\end{tabular}
\end{center}	     
\small
\begin{center}
	\begin{tabular}{|l|l|rr|rr|rr|rr|rr|}
	\hline
	  \multicolumn{2}{|r}{$\delta =$ 1E-14 \hfill Matrix:}
	 & \multicolumn{2}{c}{{\tt fidap4}}   
	 & \multicolumn{2}{c}{{\tt jagmesh8}}      
	 & \multicolumn{2}{c}{{\tt deter4}}   
	 & \multicolumn{2}{c}{{\tt plddb}}
	 & \multicolumn{2}{c|}{{\tt lp\_bnl2}}       \\  \hline	 
	 & \multicolumn{1}{|r}{\hfill RIF Time:} 
	 & \multicolumn{2}{c}{{\tt 1.5}}   
	 & \multicolumn{2}{c}{{\tt 0.5}}      
	 & \multicolumn{2}{c}{{\tt 11.0}}   
	 & \multicolumn{2}{c}{{\tt 0.4}}
	 & \multicolumn{2}{c|}{{\tt 1.6}}       \\  \hline			 	
	$k$ & Method & MV & Sec  &  MV  &  Sec & MV & Sec & MV & Sec &  MV  &  Sec \\ \hline	  
	1 &  {\footnotesize p(GD+k)} 
	     & 521 & 1.0 & 207 & 0.3 & 82 & 0.5 & 45 & 0.1 & 66 & 0.1\\ 
	1 &  {\footnotesize p(JDQMR)} 
	     & 544 & 1.0 & 229 & 0.2 & 87 & 0.4 & 45 & 0.1 & 69 & 0.1\\
	1 &  {\footnotesize JDSVD}
	     & 2037 & 10.5 & 1410 & 3.6 & 188 & 5.5 & 122 & 0.5 & 134 & 0.6\\ 
	1 &  {\footnotesize SVDIFP}
	     & 843 & 2.1 & 990 & 1.3 & 221 & 4.3 & 48 & 0.1 & 63 & 0.1 \\ \hline	     
	10 &  {\footnotesize p(GD+k)} 
	     & 2074 & 3.2 & 562 & 0.8 & 769 & 2.5 & 152 & 0.5 & 192 & 0.5\\ 
	10 &  {\footnotesize p(JDQMR)} 
	     & 2604 & 4.9 & 641 & 0.9 & 877 & 5.3 & 247 & 0.5 & 242 & 0.5\\
	10 &  {\footnotesize JDSVD}
	     & -- & -- & 12057 & 28.2 & -- & -- & -- & -- & -- & -- \\ 
	10 &  {\footnotesize SVDIFP}
	     & 7470 & 15.1 & 5024 & 6.1 & 3705 & 64.5 & 748 & 1.0 & 862 & 1.0 \\ \hline			     	 	               	     
	\end{tabular}
\end{center}
\end{table}

Tables \ref{ta: ExpC-ILU-1e-3} and \ref{ta: ExpC-RIF-1e-3} show that a 
  good preconditioner makes the problems tractable, 
  with all three methods solving the problems effectively.
Still, in most cases PHSVDS provides much faster convergence 
  and execution time on both square and rectangular matrices. 
We see that when seeking one smallest singular value with high accuracy,
  JDSVD takes less iterations for one square matrix (wang3), 
  and SVDIFP is competitive in two rectangular cases (plddb and lp\_bnl2).
This is because these cases require very few iterations, and the first
  stage of PHSVDS forces a Rayleigh-Ritz with 21 extra matrix-vector
  operations. 
This robust step is not necessary for this quality of preconditioning. 
If we are allowed to tune some of its parameters (as we did with JDSVD
  and SVDIFP) PHSVDS does require fewer iterations even in these 
  cases.

\subsection{With the shift-invert technique}
We report results on seeking 10 smallest eigenvalues with 
  PHSVDS, SVDIFP, {\tt svds}, and {\tt lansvd}. 
Because shift-invert turns an interior to a largest eigenvalue problem, 
  PHSVDS does not need the second stage.
{\tt svds} uses the augmented matrix $B$ for the shift-invert operator, 
while {\tt lansvd} computes a QR factorization of $A$.
For PHSVDS and SVDIFP we use two different factorizations, 
 an LU and a QR factorization of $A$.
All methods get the same basis size of 40, except for {\tt lansvd} which 
  is an unrestarted LBD code.
Thus, {\tt lansvd} represents an optimal method in terms of convergence,
  albeit expensive in terms of memory and computation per step.
We disable the 'COLAMD' option in SVDIFP and {\tt lansvd}, 
set the SVDIFP shifts to zero, and give a shift 1E-8 to {\tt svds}.
We have instrumented the {\tt svds} code to return the number of iterations.
To facilitate comparisons, we include the LU and QR factorization times in 
  the running times of all methods, but also report them separately. 
The tolerance is $\delta$ = 1E-10.

\begin{table}[htbp]
\centering
\caption{Seeking 10 smallest singular triplets using shift-invert.
LU(A) and QR(A) are the times for LU and QR factorizations of $A$.
The time of each method includes the associated factorization time.}
\label{ta: ExpD-shift-invert}
\small
\begin{center}
	\begin{tabular}{|l|rr|rr|rr|rr|rr|}
	\hline
	\multicolumn{1}{|r}{$\delta$ = 1E-10}
	 & \multicolumn{2}{c}{{\tt fidap4}}   
	 & \multicolumn{2}{c}{{\tt jagmesh8}}     
	 & \multicolumn{2}{c}{{\tt deter4}}   
	 & \multicolumn{2}{c}{{\tt plddb}}
	 & \multicolumn{2}{c|}{{\tt lp\_bnl2}}       \\  \hline	 
	Method & MV & Sec &  MV  & Sec & MV & Sec &  MV  &  Sec &  MV   &  Sec \\ \hline
	{\footnotesize LU(A) time}
	     & -- & 0.02 & -- & 0.01  & -- & 0.01 & --  & 0.01 & -- & 0.01  \\
	{\footnotesize PHSVDS} 
	      & 31 & 0.10 & 26 & 0.07 & 167 & 14.4 & 47  & 0.28 & 35 & 1.01  \\ 
	{\footnotesize SVDIFP} 
	      & 380 & 0.90 & 316 & 0.31 & 1177 & 168.4 & 418  & 1.92 & 432 & 9.3  \\ \hline	     
	{\footnotesize QR(A) time}
	      & -- & 0.02 & -- & 0.01 & -- & 0.53 & --  & 0.01 & -- & 0.10   \\ 
	{\footnotesize PHSVDS} 
	      & 31 & 0.29 & 26 & 0.08 & 166 & 9.1 & 27  & 0.08 & 36 & 0.48  \\ 
	{\footnotesize SVDIFP} 
	      & 383 & 2.24 & 316 & 0.37 & 1177 & 55.4 & 418  & 0.55 & 432 & 3.13  \\ \hline	
	{\footnotesize {\tt svds}}
	      & 73 & 0.33 & 61 & 0.23 & -- & -- & --  & -- & -- & -- \\ \hline
	{\footnotesize {\tt lansvd}}
	      & 31 & 0.37 & 26 & 0.24 & 133 & 3.3 & 28  & 0.23 & 34 & 0.43 \\ \hline	  
	\end{tabular}
\end{center}
\end{table}

Table \ref{ta: ExpD-shift-invert} shows that PHSVDS is faster than 
  {\tt svds} both in convergence and execution time, partly because
  it works on $C$ which is smaller in size and allows for faster convergence.
Note that {\tt svds} does not work well on rectangular matrices because
  $B$ becomes singular and cannot be inverted, and if instead 
  a small shift is used, it finds the zero eigenvalues of $B$ first. 
SVDIFP's strategy to use an inverted operator as preconditioner 
  does not seem to be as effective.
PHSVDS seems to follow closely the optimal convergence of {\tt lansvd},
  although there is high variability in execution times. 
We believe this is a function not only of the cost of the iterative method 
  but also of the different factorizations used by the two algorithms.
  
\subsection{On large scale problems}
We use PHSVDS, SVDIFP, and JDSVD to compute the smallest singular 
  triplet of matrices of order larger than 1 million. 
Information on these matrices appears in Table \ref{ta: ExpE-bigMatrix-info}.
We apply the two-stage PHSVDS on all test matrices except thermal2, 
  which is solved with dynamic PHSVDS. 
The preconditioners are applied similar to our previous experiments
  with the exception that ILU uses {\tt 'thresh=0.1'}, and {\tt 'udiag=1'}.
The tolerance is $\delta = 1E-12$. 
The symbol ``*'' means the method returns results that either did
  not satisfy the desired accuracy or did not converge to the 
  smallest singular triplet.

\begin{table}[htbp]
\centering
\caption{Basic information of some large scale matrices }
\label{ta: ExpE-bigMatrix-info}
\small
\begin{center}
    \begin{tabular}{ |c|c|c|c|c|c|}
    \hline
    Matrix 			& debr & cage14 & thermal2 & sls & Rucci1 \\ \hline
    rows $m$: 		& 1048576 & 1505785 & 1228045 & 1748122 & 1977885  \\ 
    cols $n$: 		& 1048576 & 1505785 & 1228045 & 62729 & 109900   \\ \hline
    nnz(A) 			& 4194298  & 27130349 & 8580313 & 6804304  & 7791168  \\ \hline
    $\sigma_1$ 	    & 1.11E-20  & 9.52E-2 & 1.61E-6 & 9.99E-1 & 1.04E-3  \\ \hline
    $\kappa(A)$     & 3.60E+20  & 1.01E+1 & 7.48E+6 & 1.30E+3   & 6.74E+3  \\ \hline
    Application 		& undirected & directed & thermal & Least & Least \\
    		    		& graph      &   graph &         & Squares & Squares  \\ \hline
    Preconditioner	        &  No        &  ILU(0) & ILU(1E-3) & RIF(1E-3) & RIF(1E-3) \\ \hline
    \end{tabular}
\end{center}
\end{table}

Table \ref{ta: ExpE-bigMatrix2} shows the results without or with various preconditioners.
Debr is a numerically singular square matrix. 
PHSVDS is capable of resolving this more efficiently than JDSVD,
  while SVDIFP returns early when it detects that it is not likely to 
  converge to the desired accuracy for left singular vector \cite{liang2014computing}.
All methods easily solve problem cage14 with ILU(0), but PHSVDS is much faster.
Thermal2 is an ill-conditioned matrix, whose preconditioner turns out to be less 
  effective for $C$ than for $B$.
Therefore, SVDIFP has much slower convergence than JDSVD.
Thanks to the dynamic scheme, PHSVDS recognizes this deficiency
  and converges without too many additional iterations, and with the same
  execution time as JDSVD.
However, if we had prior knowledge about the preconditioner's performance, 
  running only at the second stage gives almost exactly the same 
  matrix-vectors as JDSVD and much lower time.
Reducing further the overhead of the dynamic heuristic is part of our 
  current research.
JDSVD often fails to converge to the smallest singular value 
  for rectangular matrices since it has difficulty to distinguish them from 
  zero eigenvalues of $B$, as shown in the cases sls and Rucci1. 
For matrix sls, SVDIFP misconverges to the wrong singular triplet 
  while PHSVDS is successful in finding the correct one. 
SVDIFP and PHSVDS have similar performance for solving problem Rucci1.
In summary, PHSVDS is far more robust and more efficient than 
  either of the other two methods for large problems.

\begin{table}[htbp]
\centering
\caption{Seeking the smallest singular triplet for large scale problems. We report the time of each method including their running time and associated factorization time (PRtime) separately.}
\label{ta: ExpE-bigMatrix2}
\small
\begin{center}
	\begin{tabular}{|r|r|rrr|rrr|rrr|}
	\hline
	\multicolumn{2}{|l}{$\delta$ = 1E-12} 
	& \multicolumn{3}{c}{{\tt PHSVDS}}   
	& \multicolumn{3}{c}{{\tt SVDIFP}}
	& \multicolumn{3}{c|}{{\tt JDSVD}}        \\  \hline		  
	\multicolumn{2}{|r|}{Matrix \hfill PRtime}
		 &  MV & Sec & RES &  MV  &  Sec & RES &  MV &  Sec & RES \\ \hline
	{\footnotesize debr } 
	& --- & 539 & 84 & 3E-12  & 403* & 246* & 2E-1 & 1971 & 474.6 & 2E-12     \\ \hline
	{\footnotesize cage14 } 
	& 2E+0 & 19 & 11 & 4E-13 & 33 & 28 & 6E-13  & 111 & 185 & 7E-14   \\ \hline	  
{\footnotesize thermal2 } 
	& 3E+3 & 419 & 506 & 7E-12 & -- & -- & 4E-9  & 309 & 535 & 4E-12   \\ \hline	  	
	{\footnotesize sls }
	& 3E+3 & 1779  & 170 & 1E-09 & 408* & 328* & 1E-9  & --  & -- & 2E-0  \\ \hline	
	{\footnotesize Rucci1 }
	& 6E+4 & 4728 & 1087  & 7E-12 & 4649 & 6464 & 6E-12 & -- & -- & 5E-3 \\ \hline
	\end{tabular}
\end{center}
\end{table}

\section{Conclusion}

In this paper, we present a two stage meta-method, PHSVDS, that computes
  smallest or largest singular triplets of large matrices. 
In the first stage PHSVDS solves the eigenvalue problem on the 
  normal equations as a fast way to get sufficiently accurate approximations,
  and if further accuracy is needed, solves an interior eigenvalue problem 
  from the augmented matrix.
We have presented an algorithm and several techniques required
  both at the meta-method and at the eigenvalue solver level 
  to allow for an efficient solution of the problem.
We have motivated the merit of this approach theoretically, 
  and confirmed its performance through an extensive set of experiments. 
  
Our current implementation of PHSVDS is in MATLAB but based on top of the 
  state-of-the-art preconditioned eigensolver PRIMME. 
Thus, PHSVDS improves on convergence and robustness over other 
  state-of-the-art singular value methods, and can be used 
  on large, real world problems. 
A native C implementation of PHSVDS as part of PRIMME is planned next.

\section*{Acknowledgements}
The authors thank Zhongxiao Jia, James Baglama, Michiel Hochstenbach and Qiang Ye for generously providing their codes. The authors would also like to thank the referees for their valuable comments. This work is supported by NSF under grants No. CCF 1218349 and ACI SI2-SSE 1440700, and by DOE under a grant No. DE-FC02-12ER41890.

\bibliographystyle{siam}

\bibliography{primme_svds}

\end{document}